%% file: output.tex
\documentclass[review,11pt,authoryear]{elsarticle}

\usepackage{amssymb}
\input{config/packages}

\input{config/commands}

\journal{Arxiv}
\begin{document}

\begin{frontmatter}

\title{Optimisation models for the design of multiple self-consumption loops in semi-rural areas}

\author[inst1]{Yohann Chasseray\orcidlink{0000-0001-2345-6789}\corref{cor1}\,}
\ead{yohann.chasseray@irit.fr}
\author[inst2]{Mathieu Besançon\,\orcidlink{0000-0002-6284-3033}\,}
\author[inst3]{Xavier Lorca\,\orcidlink{0000-0002-6534-8644}\,}
\author[inst3]{Éva Petitdemange\,\orcidlink{0000-0003-3200-0446}\,}
\cortext[cor1]{Corresponding author.}

\affiliation[inst1]{organization={IRIT, Université de Toulouse, CNRS, Toulouse INP, INU Champollion},
            addressline={118 Route de Narbonne},
            city={Toulouse},
            postcode={31000},
            state={Occitanie},
            country={France}}

\affiliation[inst2]{organization={Université Grenoble Alpes, Inria, LIG},
            addressline={621 avenue centrale}, 
            city={Saint-Martin-d'Hères},
            postcode={38400}, 
            country={France}}
            
\affiliation[inst3]{organization={IMT Mines Albi},
            addressline={1 allée des sciences, Campus Jarlard}, 
            city={Albi},
            postcode={81000}, 
            country={France}}

\begin{abstract}
Collective electricity self-consumption gains increasing interest in a context where localised consumption of energy is a lever of sustainable development. While easing energy distribution networks, collective self-consumption requires complementary prosumers' profiles. Before determining the proper energy exchanges happening between these prosumers, one must ensure their compatibility in the context of collective self-consumption. Motivated by real use cases from a solar power producer, this article proposes network flow-based linear models to solve both the design aspect of the problem and the attribution of distribution flows between involved prosumers. Two main models are proposed, handling (i) single collective self-consumption loop creation and (ii) multiple loop creation. One of the objectives of this work is to provide models that can be applied at a strategic level which implies their extension to large time scale and large spatial scale. To do so, associated Benders and Dantzig-Wolfe decompositions are proposed to ensure model scalability along temporal and spatial dimensions. Proposed models are tested on realistic use cases to assess their performances.
\end{abstract}


\end{frontmatter}



\section{Introduction}
\label{sec:introduction}

In a context where the impact of global warming and the increasing scarcity of the planet's resources are being felt more and more, the management of energy resources has become a central issue and one of the major challenges our societies are facing today. Achieving a sustainable global energy production system means introducing renewable energy into the energy mix. While the share of renewable energy sources in many production systems has increased in recent decades, their internal characteristics affect the way energy is distributed. Intermittency of renewable energy sources such as wind and solar means dealing with temporally and spatially variable energy sources. In particular, since external factors, most often meteorological influence them, temporal variations can be quite unpredictable. This is especially true for photovoltaic and wind energy, for example.

Moreover, these sources have injection points that are much more decentralised and scattered across the territory than traditional sources such as nuclear or coal power plants. All these characteristics imply a degree of uncertainty which impacts national distribution networks.

All these changes force countries and energy providers to adapt their distribution schemes towards localised energy consumption. Localised consumption poses the question of a more globally decentralised energy system, as many undergoing benefits can be identified. First, achieving a localised network is a way to cut off transmission costs. In a localised network, power losses induced by long-distance transmission are minimised which would be in favour of energy efficiency \citep{kroposki2017integrating}. Reaching energy efficiency also supposes that the production of the network is tailored to the consumption of a considered area. Pleading for a decentralised network means that considered areas are revised to be local \citep{conejo2018rethinking}. As it would be easier to estimate energy needs on a local scale than on a global scale, decentralised production assures the efficiency of energy resource use. The proximity of generation and consumption sites also enhances grid stability. By reducing the risk of major disruptions, decentralisation then improves the resilience of the whole energy network \citep{mohamad2018development}. Resilience and efficiency of the network become particularly pertinent when considering the likelihood of a substantial increase in electricity consumption in upcoming decades, not only in terms of overall volume but also in terms of variability in usage (electrical cars, buses, trains). As mentioned at the beginning of this section, the network is thought to integrate renewable energies. These energies are unevenly distributed across regions. Considering local networks ensures a more accurate estimation of the involvement of renewable sources in energy generation within each specific area allowing the consideration of each area specificity \citep{lund2009effects}.

A common way to localise energy consumption is self-consumption, where produced energy is consumed onsite by its producer himself. Recent legislation (e.g.~\citet{decretacc} in France) extended self-consumption principles to collective self-consumption where producers and consumers remain locally positioned but can be represented by different legal entities. Despite this favourable context, some constraints (geographical and installed power restrictions) limit the formation of collective self-consumption operations. Besides these constraints, setting up collective self-consumption requires different but complementary consumer and producer profiles to ensure mutual interest, resulting in few existing operations.

An overarching decision-making system would be necessary to facilitate the establishment of collective self-consumption communities including solar energy companies, other producing and consuming companies and individual dwellings.
This system should consider technical, economic, and legal constraints and propose feasible and optimal loops.
Such a system will also facilitate discussions on a potential profitability of various communities.
This need for an integrated decision support system is notably formulated by Amarenco - a solar energy production company - and motivated as a real use-case our work to design optimal self-consumption loops.

In pursuit of this objective, we formulate the design of self-consumption loops as a mixed-integer linear optimisation problem. Given that the network design needs to encompass a significant time horizon and cover a relatively expansive area (spanning several cities), the proposed model has been extended with Benders and Dantzig-Wolfe decomposition techniques to address challenges on both large temporal and spatial scales.

The next section presents related work on individual and collective self-consumption optimisation with a focus on application-specific modelling choices and optimisation methods. \cref{sec:optimisation-models} details the single loop model (SLCpct) and multiloop model (MLCpct) as well as their Dantzig-Wolfe (MLCol) and Benders (SLExt, MLColExt) decompositions. \cref{sec:models-application} details the creation process of realistic instances to apply the models. This section also defines operational indicators built out of model variables and parameters that can help evaluate proposed community designs. \cref{sec:results} presents results of the application of the models on priory defined realistic instances. The performance of the models is discussed, and insights regarding the application of the models in the context of the application are provided. \cref{sec:conclusion} provides some perspectives regarding the developed models and their potential coupling with domain-related knowledge bases, leading to the conclusion and planned future works. 

\section{Related work}
\label{sec:related-work}

Despite the recent arrival of self-consumption legislation (e.g.~\cite{decretacc} in France) allowing different local actors to share a common photovoltaic production, several studies already got their interest focused on the simulation and optimisation of energy communities.

A literature review has been conducted under the criteria shown in \cref{tab:slr-criteria}, which describes two main aspects of previous studies, one being related to the size of the tackled problems and used methodology, and the other related to application-specific issues. In this systematic literature review, the objective was to find articles using optimisation techniques to ease the building of self-consumption operations that include several actors in a given community.

\begin{table}[ht]
  \centering
  \caption{Literature review evaluation criteria}
    \adjustbox{max width=\textwidth}{
    \input{tables/slr-criteria}
}%
  \label{tab:slr-criteria}%
\end{table}%

\cref{tab:slr-sumup} summarizes the conducted literature review and compares our approach to some of the analysed articles. The presented articles are selected as they encompass one or several dimensions of our approach.\\

\begin{table}[ht]
  \centering
  \caption{Comparison of analysed articles (NR: Not Relevant - NA: Not Available)}
    \adjustbox{max width=\textwidth}{
    \input{tables/slr}
}%
  \label{tab:slr-sumup}%
\end{table}%

\subsection{Specific application context}

Related lines of work show that electricity consumption optimisation is an active topic of interest. Despite some studies focusing on the specific cases of multi-energy communities, including gas, wind, and/or heat as energy sources \citep{bio_gassi_analysis_2022,gulli_recoupled_2022,pinto_optimization_2020,bokkisam_blockchain-based_2022,bahret_costoptimized_2021}, most of the identified papers deal with solar energy only, as it represents the majority of current collective self-consumption projects.

A large majority of previous work focuses on the management of self-consumption or collective self-consumption operations making the assumption that actors of loops are already defined and fixed. \cite{mustika_two-stage_2022} study the impact of the addition of a newcomer into an existing collective self-consumption loop. They propose a heuristic method to define which actor has the most impact on a self-consumption loop. Among studied articles, some only focus on individual self-consumption without considering multiple actors. Some configurations, although technically considered collective self-consumption as they involve several households, remain in a middle ground between individual and collective self-consumption. Many articles, such as \cite{gulli_recoupled_2022,mustika_new_2022,pasqui_new_2023,stephant_distributed_2021} apply their model to a group of households being located in the same building or in the same neighbourhood, which can technically be defined as a collective self-consumption loop, but remains clearly limited geographically and in terms of number of individual actors.
Geographically framed models fail to take into account the legal distance constraint that is inherent to collective self-consumption definition \citep{strepparava_privacy_2022}. More generally, regulatory constraints are rarely considered explicitly in optimisation models but are rather used as a criterion to assess found solutions \citep{pinto_optimization_2020, gribiss_configuration_2023, gil_mena_analysis_2023}.

Other lines of work extend the diversity of analysed profiles to include other types of consumers, such as professional buildings \citep{wang_peer--peer_2021,goitia-zabaleta_two-stage_2023}, electric vehicle pools \citep{surmann_agent-based_2020,piazza_impact_2023}, a winery \citep{luz_modeling_2021} or a train station \citep{simoiu_sizing_2021}, showing the economic interest of collective self-consumption.
Enlarging the diversity of profiles is in favour of collective self-consumption as diverse actors may present complementary consumption and production profiles.

Extending the breadth of considered actors also implies an extension of the considered area.
In most cases, proposed algorithms and models are applied to an urban case study \citep{brusco_renewable_2023,goitia-zabaleta_two-stage_2023,gulli_recoupled_2022,mustika_new_2022,perger_pv_2021}.
Beyond revealing that rural or semi-rural self-consumption loop configurations are rarely considered in the conception of optimisation models, it also underlines the fact that most current self-consumption loops are essentially built in urban context thus limiting the range of actors that can be considered. If urban contexts may favour collective self-consumption as the density of consumers and prosumers is higher than in rural contexts, focusing on semi-rural contexts can also have some advantages, as the constraint of geographic contiguity can be revised by regulatory texts and open to new types of producers. \cite{luz_modeling_2021} showed for instance the benefits of creating an energy community between a rural winery and a nearby city.

Extending the range of considered actors also means that proposed models should adapt to a growing number of self-consumption participants. Despite extending the diversity of involved actors, some articles focus on larger problems, defining optimisation for more than 50 actors \citep{pinto_optimization_2020,reis_collective_2022,bohringer_benefits_2023,luz_coordinating_2021}.

In a nutshell, only few previous lines of work treat the problem on both a spatial dimension (large number of actors) and a temporal dimension (large time horizon and high number of time steps). Above all, the proposed approaches are mainly focused on flow repartition, and very few of them address the design aspect of collective self-consumption loops definition. Among them, \cite{al-sorour_enhancing_2022} proposes models that include the selection of peer-to-peer household pairs, ensuring that each household is involved in only one pairing. Since this approach involves actors selection, the design context differs as the pairing is determined through semi real-time forecasting (selection is conducted on a 2-day ahead forecast), and the model’s complexity lies in its consideration of storage facilities. Moreover, the pairing selection is applied to a predefined and relatively small group of actors (6 households), which minimizes solving time concerns that typically arise in larger territories, where the selection must be performed over a greater number of actors.

\subsection{Modelling and solving approaches}

Mainly, three optimisation approaches are used, which are (i) mixed-integer linear programming models, (ii) agent-based system models and (iii) heuristics or meta-heuristics optimisation algorithms. In articles such as \cite{reis_collective_2022,stephant_distributed_2021,surmann_agent-based_2020}, agent-based models are mostly used to ensure that each participant of a self-consumption loop is making its choices independently from others and that the optimisation system is taking into account its preferences. \cite{luz_coordinating_2021} adopts the point of view of device scheduling, each device having constraints concerning its use. Agent-based programming is used to model situations where actors do not have the information about other actors to decide a strategy. It allows a compromise between the research of a group of optimal solutions while keeping constraints on each individual's preferences. In energy-sharing communities, one of the issues is the privacy regarding participant data (production and load data). As optimising a community and its associated exchanges requires knowing each participant's consumption and production profile, some work \citep{bokkisam_blockchain-based_2022} focus on systems that guarantee anonymity while allowing the use of participant's information concerning their consumption and production.

In terms of practical application, most of the articles use theoretical or simulated data to represent realistic scenarios \citep{gil_mena_analysis_2023,gribiss_configuration_2023,luz_coordinating_2021}. However, some of them are driven by a real use case involving local energy communities.

Heuristics and meta-heuristics methods are generally used to deal with larger problems in a reasonable amount of time. In the case of collective self-consumption optimisation two dimensions increase the size of problems, which are, the number of integrated producers, as well as time horizon and time step considered. \cite{wang_peer--peer_2021} for instance, use a particle swarm algorithm to optimise storage and trading strategies within an energy community involving a large number of participants (10 professional buildings and 90 households). \cite{gil_mena_analysis_2023} used the JAYA heuristic algorithm on hourly data over a one-year horizon. Similarly, \cite{heidari_physical_2021} used a genetic algorithm as the thermodynamic model used in their study is nonlinear, and becomes time-consuming for large time horizons.  

Mixed-integer linear programming models, however, remain more common methods used for collective self-consumption optimisation as exchanges between actors can easily be modelled through linear flow-type constraints. Even though they remain less costly than non-linear models, linear programming is mostly applied to small problems, which rarely exceed 10 or 20 participants on large time horizons.

The problem of collective self-consumption has to be designed over several time steps and should be considered on a more large scale than the limited cases of already existing communities involving few participants. Despite the interest of mixed-integer programming models for collective self-consumption optimisation due to their ease in modelling exchanges between actors, they are usually applied to small-scale problems involving either a few participants or a short time horizon. Addressing the complexity of collective self-consumption, however, requires a broader perspective.

As highlighted in \cref{tab:slr-sumup}, our proposal fills the identified gap in the literature regarding self-consumption loops in rural or semi-rural environments. In particular, the proposed multiloop model tends to address three aspects that have been overlooked in previous work, which are:
\begin{itemize}
    \item The integration of legal constraints directly in optimisation models when defining self-consumption loops.
    \item The inclusion of loop configuration in the optimisation problem, making it possible to build the optimal set of energy communities based on a given group of individual prosumers and their characteristics (consumption and production profile, geographical location).
    \item Model decompositions to scale to larger problems in space and in time. Very few work use both decompositions, as their combination highly complexifies the algorithms. Thus, we exploit sparsity of the neighbourhood graph to use a Dantzig-Wolfe formulation without column generation, allowing us to extend the size of the models in time and space.
\end{itemize}

\section{Optimisation models}
\label{sec:optimisation-models}

This section details the models that have been defined to find the optimal microgrid configuration for a given list of prosumers. Two main models have been defined named single loop model and multiloop model. These initial models may present a long solving time when dealing with large instances. As illustrated in \cref{fig:decompositions}, these limitations arise regarding two dimensions which are (i) growing time scale, and (ii) growing geographical scale (inducing a higher number of prosumers).

To tackle temporal scale limitations, a Benders decomposition is proposed, which splits the problem formulation into two subproblems. A first one (master problem) resolves the design aspect of the network in terms of member participation to self-consumption loops. A second one (subproblem) adjusts energy exchanges between net producers and net prosumers.

The geographical scale dimension is approached using a Dantzig-Wolfe decomposition where the column generation process is simplified through the generation of feasible loops directly from the prosumer's neighbourhood graph, avoiding the branch-and-price algorithm.

The combination of the two decompositions (Benders and Dantzig-Wolfe) on the multiloop model is then used to scale the problem-solving capacities from small numbers of actors and short time horizon to larger problems in both dimensions.

From an application point of view, the design problem can be considered at an operational level (day-to-day optimisation of self-consumption in a small neighbourhood). Our objective though is to treat strategic network design problems, as collective self-consumption is planned for several years, on potentially large territories. This changes the size of studied problems, making it necessary to scale them in time and space through presented decompositions. The combination of initial models with their decomposition led to the 5 following models, also shown in \cref{fig:decompositions}, whose performance is compared in \cref{sec:results}:

\begin{itemize}
    \item \textbf{Single loop compact model (SLCpct)} assumes that only one self-consumption loop is to be built and each actor can either join or stay apart from this self-consumption loop.
    \item \textbf{Benders extended single loop model (SLExt)} decomposes the single loop model to allow large time scale problems to be solved in a reasonable amount of time. 
    \item \textbf{Multiloop compact model (MLCpct)} extends the single loop model by providing the possibility to build several self-consumption loops with an actor joining one or none of them.
    \item \textbf{Generated cliques multiloop model (MLCol)} decomposes the multiloop model using Dantzig-Wolfe decomposition and a clique generation algorithm.
    \item \textbf{Benders extended multiloop model (MLColExt)} decomposes the Extended multiloop model to allow large time-scale problem resolution for the Dantzig-Wolfe decomposition of the multiloop model.
\end{itemize}

\begin{figure}[h]
    \centering
    \includegraphics[trim={2.1cm 0cm 2.1cm 0cm},width=1\textwidth,clip=true]{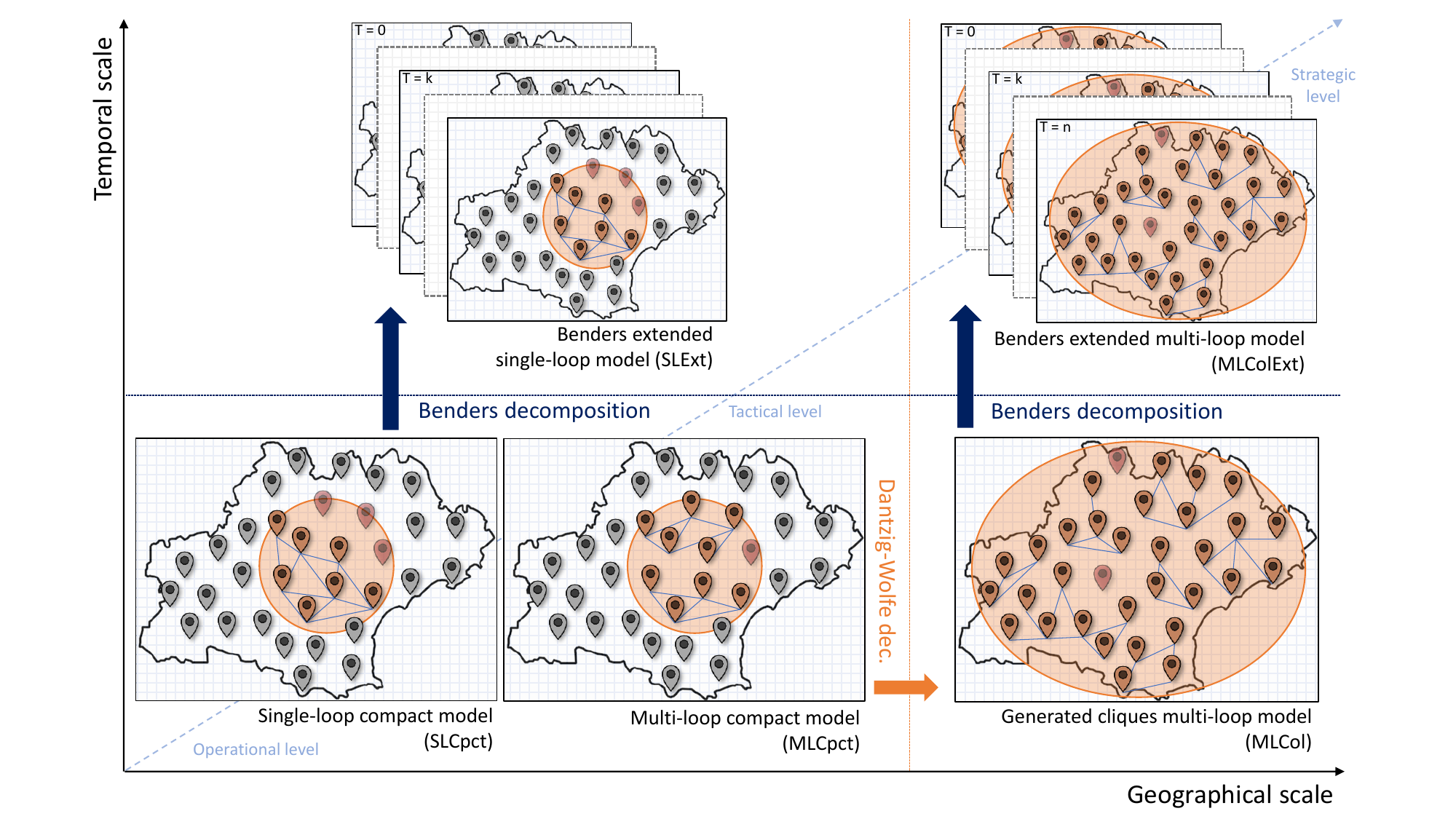}
    \caption{Objective and impact of solving approaches (Dantzig-Wolfe and Benders decomposition).}
    \label{fig:decompositions}
\end{figure}

All the presented problems embed a time scale dimension and can include a stochastic aspect to reason on a set of possible scenarios.

\subsection{Definitions and assumptions}

To express the optimisation problem, we make the following assumptions:

\begin{assumption}[Fixed actors]\label{assum:known-actors}
The list of potential members of the loops is known and fixed before the optimisation. Whether each actor belongs or not to the self-consumption loop is however still to be determined.
\end{assumption}

\begin{assumption}[External electricity pricing]\label{assum:external-price}
The electricity buying and selling cost outside the loop may vary over the actors and is proportional to the amount of bought or sold electricity. Details about price determination are given in \cref{sec:pricing}.
\end{assumption}

\begin{assumption}[Prosumers actors]\label{assum:prosumers}
Each actor is modelled as a prosumer. A producer can however have zero production as well as a consumer can have zero consumption. They are thus respectively depicted as direct consumers and direct producers.
\end{assumption}

\begin{assumption}[Consumption/production data]\label{assum:data}
The optimisation takes place over several time steps. For each time step and each scenario, the consumption and production are known for each actor. Consumption can be predicted from previous consumption data of an actor and production data can be evaluated through solar plant characteristics (installed power, panel orientation, panel tilt) and local environmental conditions (irradiation) based on the localisation of actors.
\end{assumption}

\begin{assumption}[Loop exclusivity]\label{assum:loop-exclusivity}
Single loop and multiloop configurations can be considered. In the single loop scenario, an actor either participates in the loop or is excluded from the loop. In the multiloop scenario, an actor can participate in one loop (only) among several loops or can be excluded from all of the loops.
\end{assumption}

Within the scope defined by the previous assumptions, we set the following optimisation parameters, which will help define the optimisation problem  (See~\cref{fig:graph}), \rev{and are also reported in the extensive list of notations provided in}~\ref{app:notations} :

\begin{itemize} 
\item $\mathcal{A} = \{A_{1}, \ldots, A_{n}\}$ all available actors that are potential members of a self-consumption loop.
\item $C_{1}, \ldots, C_{n}$ the net consumption of the available actors (when needed, absolute consumptions are expressed with $C_{\mathrm{abs},i}$ expression). 
\item $P_{1},  \ldots, P_{n}$ the net quantity of energy produced by each of the available actors (when needed, absolute productions are expressed with $P_{\mathrm{abs},i}$ expression).
\item $F_{1}, \ldots, F_{n}$ the grid-sourced electricity cost for each available actor.
\item $R_{1} \ldots, R_{n}$ the grid electricity selling price for each available actor.
\item $\Compactloopset = {l_{1}, \ldots, l_{n/2}}$ the set of possible self-consumption loops having at least 2 actors each, given the n available actors.
\item $d_{ij}$ is the geographical distance between two actors and $d_{\text{leg}}$ is the maximal geographical distance legally acceptable between two actors of a same self-consumption loop. \rev{We will denote by $\Edgeset$ the set of edges in the graph, and $\Neighbourset_i$ the neighbours of node $i$.}
\item $P^{\text{inst}}_{i}$ installed power of each actor and $P^{\text{inst}}_{\text{leg}}$ the installed power legal limit that should not be exceeded by the cumulative installed power of the members of a self-consumption loop.
\end{itemize}

\begin{definition}[Direct consumer]
    An actor $A_i$ is called a direct consumer when $P_i = 0$
\end{definition}

\begin{definition}[Direct producer]
    An actor $A_i$ is called a direct producer when $C_i = 0$
\end{definition}

For a natural number $n$, we use $[n] \equiv \{1 \ldots n\}$ and define the following decision variables:
\begin{itemize}
\item $x_{i} \in \{0,1\}, i \in [n]$ the binary variables indicating that actor $i$ is a member of a self-consumption loop. When extending to multiple loops, those variables are indexed by the loop: $x_i^l$. 
\item $e_{ij}, (i,j) \in [n]^2, i \neq j$ the real variables indicating the amount of energy that each actor $i$ exchange with actor $j$ within a self-consumption loop \rev{(going from actor $i$ to actor $j$)}. Self-consumed energy amounts ($e_{ii}$) are not represented as decision variables as they can be deduced afterwards.
\item $f_{1}, f_{2}, \ldots, f_{n}$ the amounts of energy supplied by the grid to each actor.
\item $r_{1}, r_{2}, \ldots, r_{n}$ the amounts of energy sold to the grid by each actor.
\end{itemize}

\begin{figure}[!h]
    \centerline{\includegraphics[trim={5.5cm 5.2cm 5.5cm 5.2cm},width=1\textwidth,clip=true]{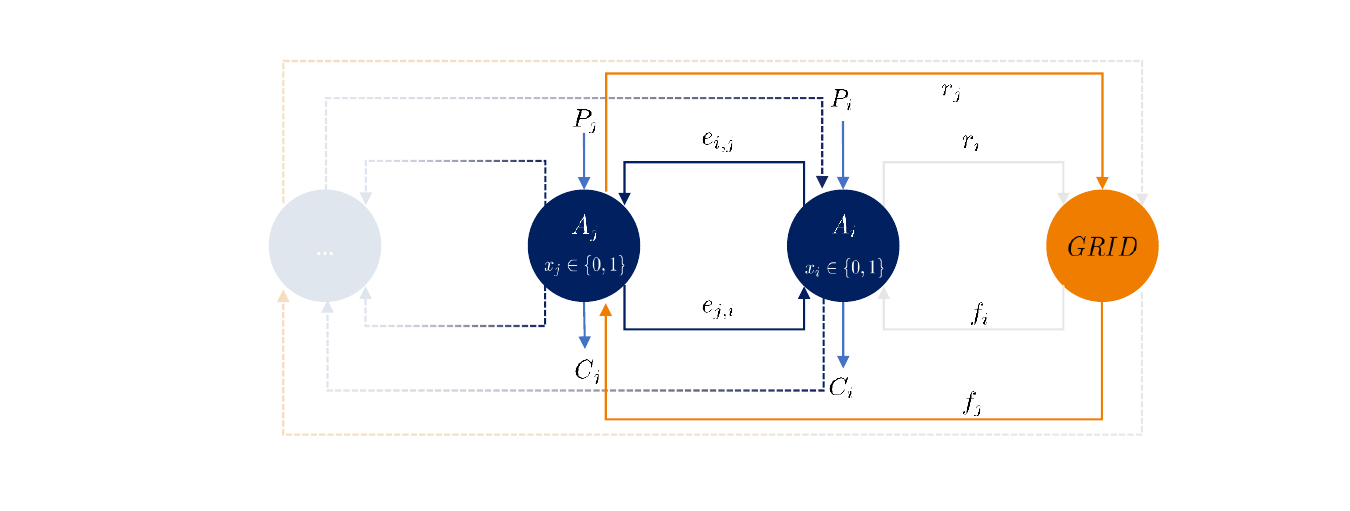}}
    \caption{Network flow model representation}
    \label{fig:graph}
\end{figure}

\subsection{Initial single loop network design and flow circulation problem}
\label{sec:single-loop-model}

In the initial flow circulation problem the first constraints to be defined are Kirchhoff flow conservation constraints, which ensure flow circulation avoiding accumulation of the energy exchanged. For each actor, flow conservation is expressed through the equality constraint between the energy amount received by the actor and the energy amount given away by the same actor. Received energy is the addition of energy received from other actors, energy produced by the actor and energy taken from the grid. Sent energy is the addition of energy given to other actors, energy consumed by the actor and energy sold to the grid. In the complete model presented below, this constraint is expressed with equation \eqref{eq:sl-kir}.

The complete single loop compact model provided as Model~\eqref{model:compact} (SLCpct) expresses the coupled problem of network design and flow repartition in a single loop configuration, meaning that each prosumer either belongs to this loop or is not participating in collective self-consumption. Details about each constraint and the objective function are given throughout \cref{sec:sc-constraints} to \ref{sec:obj-function}.

\input{models/single-loop}
$B_i^{st}$ is an upper bound on the maximum net production of actors, which can be set to (1) $\max \left\{P_i^{st} - C_i^{st}, 0\right\}$ (direct production of prosumer $i$) or (2) $P_{\mathrm{abs},i}^{st}$ (absolute production of prosumer $i$) depending on whether individual self-consumption is forced (1) or not (2). $Q^{st} = \sum_{i \in [n]} \max \left\{P_i^{st}-C_i^{st}, 0\right\}$ and  $M^{st} = \max\left\{\sum_{i \in [n]} P_i^{st}, \sum_{i \in [n]} C_i^{st}\right\}$ are big M constants used to ease the reading of the model. \rev{The model is built in a stochastic way using the set $\Scenarioset$ of scenarios. Each scenario has a probability of occurrence $p_{s}$ and its own $(R_i,F_i,C_i,P_i)$ parameters defined for each time step. Time steps are indexed by the set $\Timeset$.}

\subsubsection{Constraints from self-consumption applications}
\label{sec:sc-constraints}
To stay in line with concrete self-consumption cases, additional constraints can be added to the model. Different constraints can be set depending on the type of self-consumption that one chooses to apply, based on the following definitions:

\begin{definition}[Self-consumption]
    Self-consumption is defined at the scale of an actor. Self-consumed energy corresponds to energy that has been produced by an actor and which is consumed by this actor. Self-consumption of an actor does not imply that he takes part in a self-consumption loop. An actor $A_i$ maximises its self-consumption rate when $r_i+\sum_{j \in [n], j \neq i} e_{ij}$ is minimal.
\end{definition}

\begin{definition}[Collective self-consumption]
    Collective self-consumption is defined at the scale of a self-consumption loop. Collective self-consumed energy corresponds to the energy that has been produced by a member of the loop, and which is consumed within the loop. To do so, the produced energy can be consumed by the producer itself (self-consumption) or by another member of the loop. A loop maximises its collective self-consumption rate when $\sum_{i \in [n]} r_i$ is minimal.
\end{definition}

\paragraph{Individual self-consumption constraints}

Even though no legal requirement is considered concerning individual self-consumption, one forces the actors of a self-consumption loop to satisfy their consumption needs before exchanging energy with other actors or with the grid. To force the energy that is produced by an actor to be consumed by the same actor prosumers' production and consumption are expressed as net production and consumption, assuming that $e_{ii}$ are fixed either to $P_i$ if $P_i<C_i$ (net consumer) or to $C_i$ if $P_i>C_i$ (net producer). Individual self-consumption constraints are not related to whether or not an actor is a member of a collective self-consumption loop as even an excluded actor should consume a part of the energy they produce.

\paragraph{Collective self-consumption constraints}

The interest of participating in collective self-consumption is to favour locally produced electricity instead of grid exchanges. Each net producer must then exchange electricity with other loop members when the loop is globally in need of electricity. Similarly, net consumers must favour electricity produced inside the loop instead of electricity from the grid. In Model~\eqref{model:compact}, these constraints are expressed through Equations~\eqref{eq:sl-csc1} to \eqref{eq:sl-csc5} and equivalence \eqref{eq:sl-csc6}. \rev{To express these constraints, we introduce auxiliary variables $y_i^{st}$ and $z^{st}$, taking into account the loop priority over the grid only when the actor is part of the loop.}

\paragraph{Limiting energy circulation through actors}

The total amount of energy distributed by an actor (collective self-consumption + energy sold to the grid) should not exceed the amount of produced energy. In Model~\eqref{model:compact}, this constraint is expressed with equation \eqref{eq:sl-noncir}.
Similarly to distributed energy, the total energy amount that is consumed by an actor should not exceed the actor's consumption. However, only one of the two sets of constraints is necessary for the complete problem as Kirchhoff constraints enforce the second one.


\paragraph{Actor exclusion constraints}

When actor $i$ is excluded from the self-consumption loop, we have $x_i = 0$. As a consequence of its exclusion, an actor cannot exchange electricity with other actors. It means that $e_{ij}$ variables are subject to the value of $x_i$ and $x_j$ for each pair of actors, as depicted by Constraints~\eqref{eq:sl-actexcl1} and \eqref{eq:sl-actexcl2} of Model~\eqref{model:compact}.


\subsubsection{Maximal total installed power of a self-consumption loop}
\label{sec:power-constraint}
To ensure that a loop respects the legal installed power restriction presented previously, Constraint~\eqref{eq:sl-power} is set to prevent the global installed power from exceeding a legal criteria.

\subsubsection{Neighbourhood graph and distance restriction}

Legal distance constraints define a maximum geographical distance between two actors for them to be included in the same loop. These constraints can be represented in an undirected graph with all $n$ actors as nodes, and edges between them only if they are closer to each other than a legal distance. A loop respecting the constraints will therefore correspond to a clique in that graph.
These graphs correspond to intersection graphs of disks of identical diameters (equal to the maximum legal distance) in the plane and have been extensively studied. One noteworthy result is that maximal cliques in such graphs can be generated in polynomial time \citep{clark1990unit}.
Additionally, edges can be removed in preprocessing between two actors if the sum of the installed power production exceeds the allowed threshold (see \cref{sec:power-constraint}).

To cope with the legal restrictions concerning the geographical distance between each pair of actors participating in a self-consumption loop ($d_{ij}<d_{leg}$) equation \eqref{eq:sl-conflictgraphcons} refers to $\Edgeset$ to avoid non-neighbours to be part of the same collective self-consumption loop and define associated constraints.

\subsubsection{Uncertainty in the production and consumption profiles}

Renewable energy production and consumption are by nature uncertain and potentially highly correlated across actors and time steps. Furthermore, the contracts of certain actors with the grid may involve variations in the buying or selling prices, denoted as $F_i$ and $R_i$ respectively. These variations could be implemented through schemes such as time-of-use pricing or could involve uncertainties in prices, such as when actors participate in price-based demand response programs like peak shaving \citep{Paterakis2017,Albadi2008}. We therefore consider $(R_i^t,F_i^t,C_i^t,P_i^t)$ to be non-independent random vectors observed through a finite set of scenarios \Scenarioset with given associated probabilities $p_s$.

We model the problem as a two-stage stochastic program, with the loop membership decisions $x$ taken at the first stage, and the flow variables at the second stage. A scenario $s$ corresponds to a realisation of the uncertain parameters for all time steps, after which the loop operator can decide on the continuous flow variables $e, f, r$ in an operational stage. Importantly, the decisions taken at all time steps do not influence the state of the system, nor the decisions for the following time steps. This allows us to consider only these two design and operation stages, with the operation stage encompassing all time steps at once.

\subsubsection{Objective function}
\label{sec:obj-function}
In this optimisation problem, the objective is to find the best self-consumption loop from an economic point of view, that is to say, favouring exchanges between the grid and actors who have the most interesting prices  (high selling prices, low buying prices)\footnote{The objective function takes all actors into account, without considering whether they belong to the self-consumption loop. This is used to include the trivial case where no actor belongs to the self-consumption loop, which would result in an objective value computed on loop members only to be equal to zero.} Fixed inclusion costs ($T_i$ adhesion costs of each actor $i$), as well as the price of the electricity within the loop ($r$), are not considered in the objective function, as they will be determined once the geometry of the loop and exchange happening inside the loop are set. Fixing $T$ after this first optimisation step to reach equity between actors involved in the self-consumption loop is evoked as a perspective of this work (\cref{sec:conclusion}). Despite balancing the benefits of the actors is left for future work, the proposed model assures a maximal global benefit and a non-zero benefit for each actor involved in the self-consumption loop.

\subsection{Benders decomposition for larger time scales (SLExt)}
\label{sec:benders-single-loop-model}

We can notice that the model is structured in two parts, a choice of selected actors that form a feasible loop, and auxiliary binary variables on one hand, and a set of continuous flow decisions for each scenario and time step, which has a block-separable structure along these two dimensions.

We can use a Benders decomposition of the problem, expressing the flow optimality conditions as constraints on the $x$ and $z$ variables. Benders decomposition exploits the separation of variables between the binding variables which design the network and the continuous variables that are fully block-separable (see \cite{rahmaniani2017benders} for a recent literature review of the Benders decomposition, applications and algorithms).
We can note $\phi_{st}: \{0,1\}^n \times \{0,1\} \mapsto \mathbb{R}$ the optimal value function of the problem given fixed values of $x$ and $z^{st}$. Model~\eqref{model:compact} is then equivalent to \rev{Model}~\eqref{model:equivalentmodel}:
\begin{subequations}\label{model:equivalentmodel}
\begin{align}
\min_{x,z,\eta} \, &\sum_{s\in\Scenarioset} \sum_{t\in\Timeset} p_s \eta_{st} &\\
\mathrm{s.t. } \, & \eta_{st} \geq \phi_{st}(x,z^{st}) & \forall s \in \Scenarioset, t \in \Timeset\label{cons:optimalityprimalbenders}\\
& \eqref{eq:sl-conflictgraphcons}-\eqref{eq:sl-compactbinvars}
\end{align}
\end{subequations}
\rev{where $\phi_{st}(x,z^{st})$ is the solution of Model}~\eqref{model:continuousprojection} \rev{defined below}. At the optimal solution, the variable $\eta_{st}$ will match the value of $\phi_{st}(x,z)$. The value of $\phi$ corresponds to the optimisation problem once the binary variables are fixed.
This problem is entirely decomposed along $(s,t)$, we, therefore, omit the two indices for clarity.
The dual variable associated with each constraint is labelled on the right:
\begin{subequations}\label{model:continuousprojection}
\begin{align}
\phi(x,z) = \min_{e,f,r,y} \, & \sum_i F_i f_i - R_i r_i & \\
\mathrm{s.t. }\, & e_{ij} \leq x_i P_i & (\alpha_{ij}^p) \\
& e_{ij} \leq x_j C_j & (\alpha_{ij}^c) \\
& \sum_j e_{ij} + r_i - \sum_j e_{ji} - f_i = P_i - C_i & (\pi_i) \\
& y_i \leq x_i Q & (\gamma_i) \\
& y_i - r_i \leq  0 & (\theta_i) \\
& y_i - r_i \geq -(1-x_i) Q & (\kappa_i) \\
& \sum_i y_i \leq \sum_i x_i (P_i - C_i) + M (1-z) & (\nu) \\
& y_i \leq M z & (\lambda_i) \\
& \sum_j e_{ij} + r_i \leq B_i & (\tau_i)\\
& e,f,r,y \geq 0.
\end{align}
\end{subequations}

Our problem is always bounded since the energy traded within the grid is bounded by the productions, and is always feasible, since there always exists a baseline solution in which all actors trade with the grid only \footnote{The dual of Model~\eqref{model:continuousprojection} is provided in~\cref{app:dual_benders} (Model ~\eqref{model:dualcontinuousprojection}).}. This implies that $\phi_{st}(x,z)$ is finite for all $x$ and $z$ and that the primal and dual problems attain the same optimal value by the strong duality of linear problems. The feasible region of the dual Model being independent of $x$ and $z$, we can replace Constraint~\eqref{cons:optimalityprimalbenders} with:
\begin{align*}
    \tag{B-CUT} & \eta_{st} \geq \\
    &\;\; \sum_{i,j\in\Edgeset} (\alpha_{ij}^p x_i P_i + \alpha_{ij}^c x_j C_j) \\
    &\;\; + \sum_i -(1-x_i) M \kappa_i + (P_i - C_i) \pi_i + M z \lambda_i + B_i \tau_i + \gamma_i x_i Q \\
    &\;\; + \nu (\sum_i x_i (P_i - C_i) + (1-z) M) \,\,\,\forall (\alpha^p, \alpha^c,\pi,\gamma,\theta,\kappa,\nu,\lambda,\tau) \in \Vertexset_{st}
\end{align*}
where $\Vertexset_{st}$ is the set of vertices of the dual feasible set. The cardinality of $\Vertexset_{st}$ grows exponentially with $n$, prohibiting the addition of all inequalities in (B-CUT), and motivating their generation on the fly to separate candidate solutions.

\subsection{Introduction of multiple loops (MLCpct)}

In this section, we extend the proposed model from a single loop in which actors can be included or not to multiple loops defined on the fly. This extension will be crucial to scale the proposed model in space, i.e.~allowing the loop operator to build and operate multiple loops defined dynamically across a broad area and considering many actors simultaneously.

We define a compact representation of the model\footnote{The full Model~\eqref{model:multiloop-cpct} can be found in~\ref{app:multiloop}} and propose an extended formulation based on a Dantzig-Wolfe decomposition. We claim the extended formulation can be solved directly and not through a branch-and-price algorithm, with a number of variables remaining small enough as the number of actors increases, and support that claim with the characteristics of potentially optimal loops.

The number of non-trivial loops is bounded above by $n/2$, we will denote $l \in \Compactloopset$ the loop indices. The loop membership variables $x_{i}^{l}$ are indexed by the loop and takes value $1$ when actor $i$ is assigned to loop $l$. Aside of $x_{i}^{l}$ variables, some of the constraints defined in \cref{sec:single-loop-model} evolve as soon as multiple loops are considered. Exclusion constraints (Constraints~\eqref{eq:sl-actexcl1} and \eqref{eq:sl-actexcl2}) should be rewritten as follows, to ensure that two actors not being members of the same loop cannot exchange electricity:

\begin{equation}\label{eq:multi-loop-actors-exclusion}
        e_{ij} \le \sum_{l \in \Compactloopset}w_{ij}^{l}D_{i,j}, \forall (i,j) \in \{1, \ldots n\}^2, i \neq j,
\end{equation}
where $D_{i,j} = min(P_i, C_j)$. To define this constraint, we introduced $w_{ij}^{l}$ which are binary variables linearising $x_{i}^{l} x_{j}^{l}$, expressing the membership of both actors $i$ and $j$ to loop $l$ and constraining the flow through equation \eqref{eq:multi-loop-actors-exclusion}. These $w_{ij}^{l}$ variables are subject to:
\begin{align*}
        & w_{ij}^l \le x_i^l, \forall (i,j) \in \{1, \ldots n\}^2, i \neq j,l \in \Compactloopset \\
        & w_{ij}^l \le x_j^l, \forall (i,j) \in \{1, \ldots n\}^2, i \neq j,l \in \Compactloopset.
\end{align*}

Legal restriction constraints, collective self-consumption constraints and circulation constraint remain unchanged, and are now indexed by the loop $l$.
Additionally, big M constraints are redefined to allow the $y_i^l$ variables to take any value when actor $i$ is included in loop $l$.

 \subsection{Dantzig-Wolfe decomposition (MLCol)}

The previous multiple loop formulation is amenable to a Dantzig-Wolfe decomposition exploiting the structure of solutions to the knapsack constraint~\eqref{eq:knapsackmultiloop} under a loop maximality assumption.

\begin{assumption}[Loop maximality]\label{assum:maximality}
An optimal solution exists such that all created loops are maximal, in the sense that for any actor $i$ not included in any loop, it could not be added to any of them without removing another actor from the loop.
\end{assumption}

\begin{proposition}\label{prop:maximality}
\cref{assum:maximality} holds for multiloop compact model ~\eqref{model:multiloop-cpct}.
\end{proposition}
\begin{proof}
Given a non-maximal solution, let us assume there exists an actor that could be included in a loop without removing another one. We can construct a solution in which this actor does not receive any net consumed power from the grid and sells any net produced power to the grid, thus not modifying the objective.
\end{proof}

\Cref{prop:maximality} holds because of the feasibility of a trivial solution where an actor could be included in a loop without contributing to internal flows. In practice, such an actor could be removed from the loop in a post-processing step without degrading the objective. This leads us to observe that an optimal solution can always be constructed by selecting a subset of maximal loops which do not overlap on any actor.

The number of potential loops could grow exponentially with the number of actors, and would therefore require a column generation approach to build new loops. We will rely on a spatial sparsity assumption, together with the optimality of maximal loops to motivate an extended formulation that eagerly generates all extended columns before branching.

Finding maximal loops corresponds to finding cliques in the graph of all actors with an edge between $i$ and $j$ if $d_{ij} \leq d_{\text{leg}}$. These cliques are not maximal since they also need to satisfy the installed capacity constraint.

Furthermore, each selected clique needs to contain at least one producer and one consumer to generate some inner flows, further reducing the number of potential clique candidates. Since there will typically be few producers compared to the number of consumers, we can construct cliques around each of the producers only. The procedure we propose is:
\begin{enumerate}
    \item Generate all maximal cliques in the graph, starting net producers only,
    \item If a clique respects the installed capacity, it can be used directly,
    \item Otherwise, enumerate all maximal vertices of the 0-1 knapsack polytope selecting the producers of the clique, subject to the capacity constraint. Each vertex results in a combination of producers, and all consumers of the clique can be added to the loop.
\end{enumerate}

We can adapt the clique generation algorithm to dynamically forbid adding nodes that would exceed the capacity, or generate all maximal loops from every clique by excluding actors until the capacity constraint is respected.

The extended model is described by a set $\Extendedloopset$ of maximal loops and binary variables $v_h$ capturing whether a loop is taken, with parameters $a_{i}^{h}$ being equal to $1$ if loop $h$ contains actor $i$. Furthermore, additional application-related constraints can directly be included in the loop, including, for instance, the fact that two coupled actors $(i,j) \in \Directset$ are both present or both absent from a loop, \rev{$\Directset$ being the set of coupled actors}. The extended formulation of the multiloop model~\eqref{model:multiloop-cpct} is then essentially captured by the new constraints:
\begin{subequations}
\begin{align}
& \sum_{h\in\Extendedloopset} a_i^h v_h \leq 1& \forall i\in[n]\label{eq:packingextended}\\
& e_{ij}^{st} \leq D_{ij}^{st} \sum_{h\in\Extendedloopset} v_h a_i^h a_j^h  & \forall (i,j) \in \Edgeset, s\in\Scenarioset, t\in\Timeset. \label{eq:edgeblocker}
\end{align}
\end{subequations}

Constraint~\eqref{eq:packingextended} of the extended model corresponds to Constraint~\eqref{eq:ml-oneloop} in the compact one, i.e. each actor is at most in one loop. Constraint~\eqref{eq:edgeblocker} replaces the $w$ variables and associated auxiliary constraints of the compact model, i.e. two actors can share electricity only if they belong to the same loop. The full decomposed model is provided in~\ref{app:multiloop-danzig-wolfe}.\\

\subsection{Benders decomposition of MLCol model (MLColExt)}

Similarly to the decomposition operated on the single loop model in \cref{sec:benders-single-loop-model}, the extended version of the problem can be split into two subproblems, that can be efficiently solved using Benders' algorithm. As previously, the main problem solves the design aspect of the network, by setting the loops that should be selected for the collective self-consumption operation. One can keep the notation $\phi_{st}: \{0,1\}^n \times \{0,1\} \mapsto \mathbb{R}$ the optimal value function of the problem given fixed values of $v,\eta_{st}$ to define this main model:
\begin{subequations}
\begin{align}
\min_{v,\eta} \, & \sum_{s\in\Scenarioset} \sum_{t\in\Timeset} p_s \eta_{st} & \\
\mathrm{s.t. } \, & \eta_{st} \geq \phi_{st}(v) & \forall s \in \Scenarioset, t \in \Timeset\label{cons:optimalityprimalbendersExtended}\\
& \sum_{h\in\Extendedloopset} a_i^h v_h \leq 1 & \forall i \in [n]\\
& v \in \{0,1\}^{|\Extendedloopset|}.
\end{align}
\end{subequations}

The second problem, using solutions of the main problem as parameters, is then a continuous problem whose resolution should be possible in a reasonable amount of time as no binary variables are involved. \rev{In this formulation, we define $\Excessset_{st}$ as the set of loops having a positive net production for a given $(s,t)$ tuple.} We drop the indices $s$ and $t$ from the subproblem below, as it is fully separable along $\Scenarioset$ and $\Timeset$. The dual variable associated with each constraint is labelled on the right.
\begin{subequations}\label{prob:subproblemExtended}
\begin{align}
\min_{e,r,f} & \sum_{i\in [n]}\, (F_i f_i - R_i r_i ) & & \\
\text{s.t. } & e_{ij} \leq D_{ij} \sum_{h\in\Extendedloopset} v_h a_i^h a_j^h & \forall i \in [n]\;\;& (\alpha_{ij}) \\
& \sum_{j\in\Neighbourset_i} e_{ij} + r_i - f_i - \sum_{j\in\Neighbourset_i} e_{ji} = P_i - C_i & \forall i \in [n]\;\; & (\pi_i) \\
& r_i \leq (1-\sum_{h\notin\Excessset} v_h a_i^h) Q & \forall i \in [n]\;\; & (\gamma_i) \\
& \sum_{i\in[n]} r_i a_i^h \leq \sum_{i\in [n]} a_i^h v_h (P_i - C_i) + Q (1-v_h) & \forall i \in [n],h\in \Excessset\;\; & (\nu_h) \\
& \sum_{j\in \Neighbourset_i} e_{ij} + r_i \leq P_{i} & \forall i \in [n]\;\; & (\tau_i) \\
& e, r, f \geq 0.
\end{align}
\end{subequations}

\subsubsection{Dual and cuts}
Dualisation of Problem~\eqref{prob:subproblemExtended} gives the following model, whose objective value provides cuts that can be applied to the main model to refine the chosen design of collective self-consumption loops:
\begin{align*}
\max_{\alpha,\pi,\gamma,\kappa,\nu,\tau} \, & \;\; \sum_{i,j\in\Edgeset} \alpha_{ij} D_{ij} \sum_{h\in\Extendedloopset} v_h a_i^h a_j^h + &\nonumber\\
    &\;\; + \sum_{i\in[n]}\big( (P_i - C_i) \pi_i + P_i \tau_i + \gamma_i (1-\sum_{h\notin\Excessset} a_i^h v_h) Q \big) &\nonumber \\
    &\;\; + \sum_{h\in\Excessset} \nu_h \Big( \sum_{i\in[n]} \big( a_i^h v_h (P_i - C_i) \big) + Q (1-v_h) \Big) &\label{eq:dualobjective2} \\
\mathrm{s.t. }\,\, & \alpha_{ij} + \pi_i + \tau_i \leq 0 && (e_{ij})  \\
                   & \pi_i + \gamma_i + \tau_i + \sum_{h\in\Excessset} a_i^h \nu_h \leq -R_i && (r_{i})  \\
                   & -\pi_i \leq F_i && (f_{i})  \\
& \alpha, \gamma, \nu, \tau \leq 0, \kappa \geq 0.
\end{align*}
These dual-induced cuts can be expressed as follows:
\begin{align*}
    \tag{EXT-B-CUT} & \eta_{st} \geq \\
    &\;\; \sum_{i,j\in\Edgeset} \alpha_{ij} D_{ij} \sum_{h\in\Extendedloopset} v_h a_i^h a_j^h + &\nonumber\\
    &\;\; + \sum_{i\in[n]}\big( (P_i - C_i) \pi_i + P_i \tau_i + \gamma_i (1-\sum_{h\notin\Excessset} a_i^h v_h) Q \big) &\nonumber \\
    &\;\; + \sum_{h\in\Excessset} \nu_h \Big( \sum_{i\in[n]} \big( a_i^h v_h (P_i - C_i) \big) + Q (1-v_h) \Big) \,\,\,\forall (\alpha,\pi,\gamma,\kappa,\nu,\lambda,\tau) \in \Vertexset_{st}.
\end{align*}

\section{Application of the models}
\label{sec:models-application}
This section details the data and assumptions used to set up realistic instances and their corresponding consumption and production profiles. The method to generate these instances is designed to mimic real-case scenarios. These instances are then reused in a test protocol on operational-specific metrics, which is described in \cref{sec:test_protocol}.

\subsection{Generation of realistic instances}
\label{sec:case-study}
To apply the proposed model to realistic instances, localization, costs, production, and consumption data have been generated. Each actor is then represented as a prosumer, capable of both producing and consuming electricity.

\subsubsection{Consumption load}
\label{sec:consumption-load}
Consumption profiles are estimated from reference datasets extracted from the Enedis\footnote{Enedis is a state-owned limited company managing and developing the electricity distribution network in France.} Open Data platform \citep{enedis-open-data}. Consumption loads are described with a 30-minute time step over a year (2022 is taken as a reference).

To generate a group of consumers, we divided this group into two consumption sub-groups, namely "Household" and "Pro: Professional." These categories echo the separation made by the Enedis Open Data platform, which provides an average profile of each category. A distribution of consumers in each category is manually determined based on the scenario chosen for simulation (see \cref{sec:test_protocol}). All consumers in each category are then provided with a consumption profile that is normally distributed around the Enedis reference profile load corresponding to that category. To do so, a rolling window mean using a 3 steps window size is applied. Some variations are added to this averaged consumption using a uniform random variable and a variation factor of 30\%.

\subsubsection{Installations characteristics and production volume computing}
Determining the actual production of an installation is complex as it depends on several parameters, some of which are meteorological. To calculate discretised production, irradiance data is computed based on solar position, installed power of the installation, inclination, and orientation. Two exposition configurations are considered: 
\begin{itemize}
    \item Worst exposition (WC): Tilt angle fixed to 60° and panels facing North (0° azimuth)
    \item Best exposition (BC):  Tilt angle fixed to 30° and panels facing South (180° azimuth)
\end{itemize}

To avoid longer computation runtimes, solar position information is computed on 10 km² surface tiles and reused identically for each installation associated with this tile. This assumes that solar position variation along such a surface has a low impact on the installation production. Based on the installed power assigned to each producer, the production is generated for each time step using PVGIS-computed irradiance data.

\subsubsection{Electricity pricing}
\label{sec:pricing}

The objective function defined in every proposed problem is subject to electricity pricing conditions. \cref{tab:selling} and \cref{tab:buying} report official buying costs \citep{buying-cost} and French Energy Regulatory Commission selling prices \citep{cre-selling} of the electricity for actors engaged in a self-consumption operation.

\begin{table}[ht]
  \centering
  \caption{Electricity selling prices considering self-consumption installed power \citep{cre-selling}}
    \adjustbox{max width=\textwidth}{
    \begin{footnotesize}
    \input{tables/selling_prices}
    \end{footnotesize}
}%
  \label{tab:selling}%
\end{table}%

\begin{table}[ht]
  \centering
  \caption{Electricity buying cost considering the client consumption subscription profile  \citep{buying-cost}}
    \adjustbox{max width=\textwidth}{
    \begin{footnotesize}
    \input{tables/buying_prices}
    \end{footnotesize}
}%
  \label{tab:buying}%
\end{table}%

\subsubsection{Geographical distribution}
One of the characteristics of the proposed model is its geographical dimension. To construct representative use cases, two types of geographical distributions have been defined:

\begin{itemize}
    \item \textbf{Uniform distribution}: In this distribution, each installation is located uniformly, maintaining a global density of actors. This random distribution is achieved by utilizing a combination of uniform distributions along latitudes and longitudes.
    \item \textbf{Clustered distribution}: This distribution is used to allocate groups of installations, imitating how industrial installations are distributed in semi-rural territories. In our case, clusters are formed having between 4 and 6 actors each.
\end{itemize}

In the following generated instances, these distributions are consistently utilized to sample a group of actors with a global density parameter. Multiple distinct densities are employed to assess the impact of this parameter on discovered solutions.

\subsection{Test protocol}
\label{sec:test_protocol}
The following test protocol has been established to validate the growing complexity of the defined optimisation problem. Two parameters are initially observed, namely, time horizon and the number of actors as they influence the initial number of variables defined in the problem. Other parameters of the initial configuration (named \textit{reference}) are detailed in \cref{tab:plan}, which also describes other instances that are used to compare optimisation results for different configurations.

\begin{table}[ht]
  \centering
  \caption{List of instances and their characteristics}
    \adjustbox{max width=\textwidth}{
    \begin{footnotesize}
    \input{tables/plan}
    \end{footnotesize}
}%
  \label{tab:plan}%
\end{table}%

Based on Enedis consumption profiles, three prosumer profiles are defined, namely \textit{Household}, \textit{Pro1} and \textit{Pro2} whose characteristics are described in \cref{tab:profiles_test}. From these profiles, a scenario has been defined, representing a distribution of the consumption profiles among generated actors (rate column). The objective of the defined scenario is to represent both the small producers that can be either dwellings (Household profile) or professional actors (Pro1 profile) and larger producers that are essentially professional actors (Pro2 profiles). \rev{In France, the number of photovoltaic installations has reached over 760,000, mainly represented by individual households (80\% of installations are under 6kW) }\citep{numb-selfcons}\rev{. However, since self-consumption often involves professional actors as well }\citep{numb-selfcons}\rev{, our instances balance this distribution by attributing more representative weights to larger prosumers.}

\begin{table}[ht]
  \centering
  \caption{Description of the prosumer profiles used in the base configuration}
  \begin{footnotesize}
      
    \adjustbox{max width=\textwidth}{
    \input{tables/profiles_test}
    }
  \end{footnotesize}
  \label{tab:profiles_test}%
\end{table}%

Instances involving 10 actors and having the same configuration as the reference instance presented in \cref{tab:plan} are also solved for different time horizons to evaluate how the models behave in solving larger problems. With the same objective, instances having several number of prosumers and a time horizon of 7 days are solved.

\cref{tab:instance_desc} indicates the different sizes - in terms of number of variables and constraints for all instances averaged on 100 replicates for each of them.
 
\begin{table}[ht]
  \centering
  \caption{Instances sizes (average computed on 100 replicates for each configuration and model) \rev{expressed in terms of number of constraints, total number of variables and number of instance variables. MLCol : Generated cliques multiloop model - SLCpct : Single loop compact model - MLCpct : Multiloop compact model}}
  \begin{footnotesize}
      
    \adjustbox{max width=\textwidth}{
    \input{tables/instances_size}
    }
  \end{footnotesize}
  \label{tab:instance_desc}%
\end{table}%


\subsection{Evaluation metrics}

Each model presented in \cref{sec:optimisation-models} is applied to generated instances. Aside from objective value and solving time, results are analysed through several evaluation metrics defined below. These \rev{key performance indicators} are only used for a post optimisation analysis of the loops, as benefit balance within the loops is discussed but left for future work.

\subsubsection{Collective self-consumption and self-production rates}
Self-consumption rate is defined as the ratio between the energy that is consumed within the collective self-consumption loop and the total amount of produced electricity within the loop. Self-production rate is defined as the ratio between the energy that is consumed within the collective self-consumption loop and the total amount of consumed electricity within the loop.

\subsubsection{Number of proposed loops}
The number of proposed loops indicates the number of collective self-consumption loops that have been formed and is computed using the decision variables' obtained values: 
\begin{subequations}\label{eq:number_of_loops}
\begin{align}
& & N_{\text{loop}} = \sum_{h\in\Extendedloopset} v_h & & \text{for extended models} & \\
& & N_{\text{loop}} = \sum_{l\in\Compactloopset} \delta_l & & \text{for compact models}, &
\end{align}
\end{subequations}
where $\delta_l$ is defined as follows:
\begin{align*}
\delta_l = \begin{cases} 1 & \text{if $\sum_{i \in [n]} x_{i}^{l} > 0$ (at least one prosumer selected in the loop)} \\ 0 & \text{otherwise.}\end{cases}
\end{align*}

For clarity in the definition of later indicators, we only consider the case of extended multiloop (MLColExt) model only. Nevertheless, the same indicators can also be built for compact multiloop models, indexing $x_i^l$ with $l\in\Compactloopset$ and $i\in [n]$ instead of $v_ha_i^h$ ($h\in\Extendedloopset$).

\subsubsection{Fairness indicators}
\label{sec:equity-indicator}
Two indicators representing the disparity of the benefit distribution are defined. The highest benefit indicator \eqref{eq:equity-1} underlines the highest benefit realised among all actors reaching a collective self-consumption community. This benefit is computed relative to a situation where no loop has been defined (zero loop configuration). Conversely, the lowest benefit indicator \eqref{eq:equity-2} underlines the lowest benefit realised among all actors proposed to reach a collective self-consumption loop.
\begin{subequations}\label{eq:equity}
\begin{align}
& \text{Highest benefit:}& & \overline{b} = \max_{i\in [n]}\, (\sum_{s\in\Scenarioset} \sum_{t\in\Timeset} p_s\, (F_i^{st} (f_i^{st}-C_i^{st}) - R_i^{st} (r_i^{st}-P_i^{st}))) \label{eq:equity-1}\\
& \text{Lowest benefit:}& & \underline{b} = \min_{i\in [n]}\, (\sum_{s\in\Scenarioset} \sum_{t\in\Timeset} p_s\, (F_i^{st} (f_i^{st}-C_i^{st}) - R_i^{st} (r_i^{st}-P_i^{st}))). \label{eq:equity-2}
\end{align}
\end{subequations}

\subsubsection{Average number of members and installed power per loop}

The number of created loops can be put into perspective with the number of actors that are not participating in any loop, by computing the average number of actors per loop :
\begin{equation*}
N_{\text{act\_mean}} = \frac{1}{N_{\text{loop}}} \sum_{h\in\Extendedloopset} \sum_{i \in [n]} v_h a_i^h
\end{equation*}

Another indicator is the average installed power per loop, which underlines how complete loops are in terms of installed power:
\begin{equation*}
P_{\text{mean}} = \frac{1}{N_{\text{loop}}} \sum_{h\in\Extendedloopset} \sum_{i \in [n]} v_h a_i^h P_i^{\text{inst}}.
\end{equation*}
This indicator should always be analysed in relation to the legal installed power limit set initially to solve the problem.

\subsubsection{Total benefit realised through self-consumption}
This indicator is obtained from realised benefits by each prosumer who is involved in a self-consumption loop relative to a zero loop configuration:
\begin{equation}\label{eq:benefit}
\sum_{h\in\Extendedloopset} \sum_{i \in [n]}  \sum_{s\in\Scenarioset} \sum_{t\in\Timeset} p_s v_h a_i^h \, (F_i^{st} (f_i^{st}-C_i^{st}) - R_i^{st} (r_i^{st}-P_i^{st})).
\end{equation}

\subsubsection{Root-node gap}
\rev{In order to analyse the tightness of each problem formulation, we used the relative root-gap defined as the objective function value difference between the best known integer solution before the branch-and-cut algorithm and the linear relaxation of the problem.}

\section{Results and discussion}
\label{sec:results}

\rev{This section discusses the obtained results in terms of computation times and business-related observations over the found solutions. All models were implemented in Python using the PuLP library, and the resulting instances were solved with the Gurobi optimisation solver. The LP files of all generated instances are available at} \citet{chasseray_2025_15074668}.

\subsection{Scalability on the temporal and geographical dimensions}


\begin{figure}[h]
    \centering
    \begin{subfigure}[b]{0.49\textwidth}
        \centering
        \includegraphics[width=\textwidth]{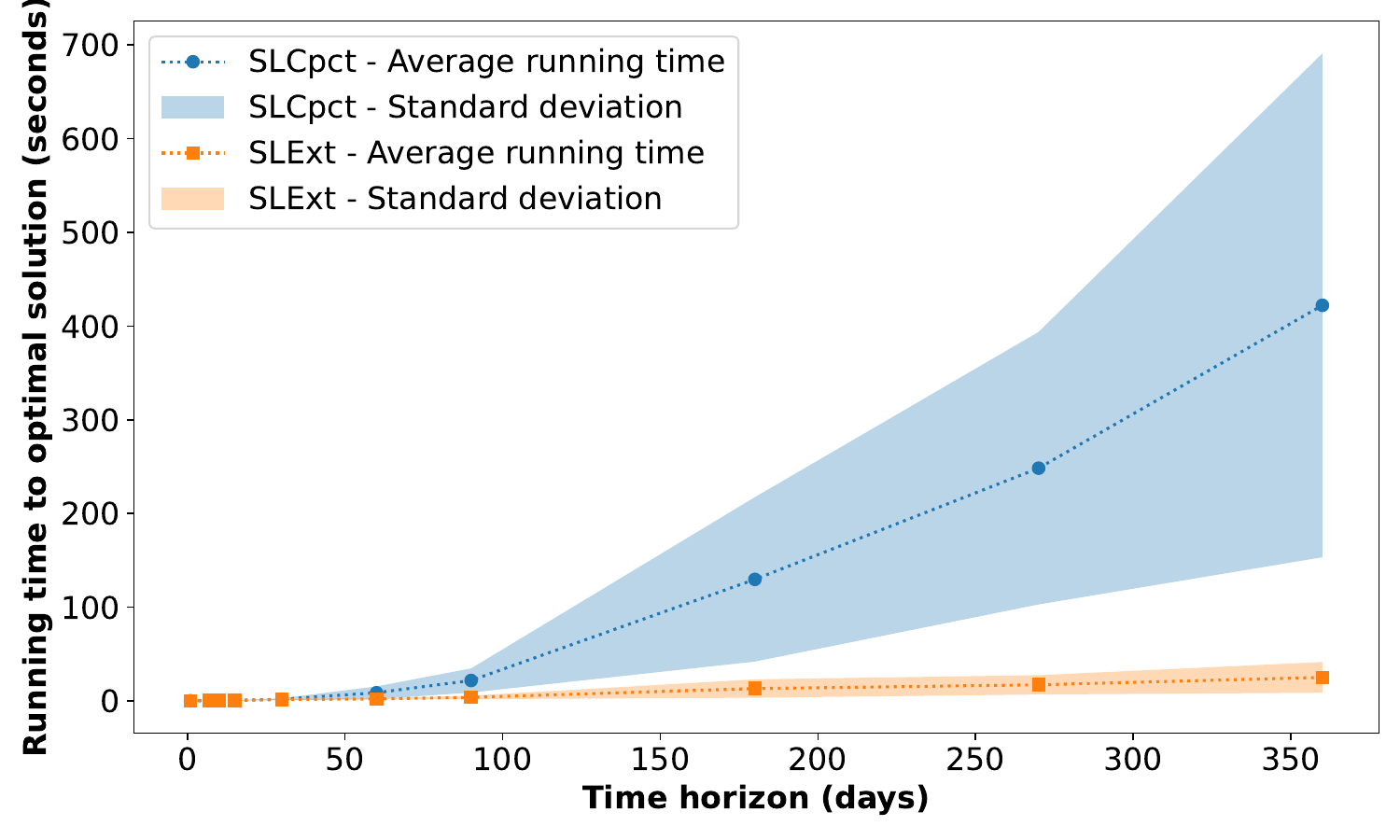}
        \caption{Across time horizon (10 actors)}
        \label{fig:horizon}
    \end{subfigure}
    \begin{subfigure}[b]{0.49\textwidth}
        \centering
        \includegraphics[width=\textwidth]{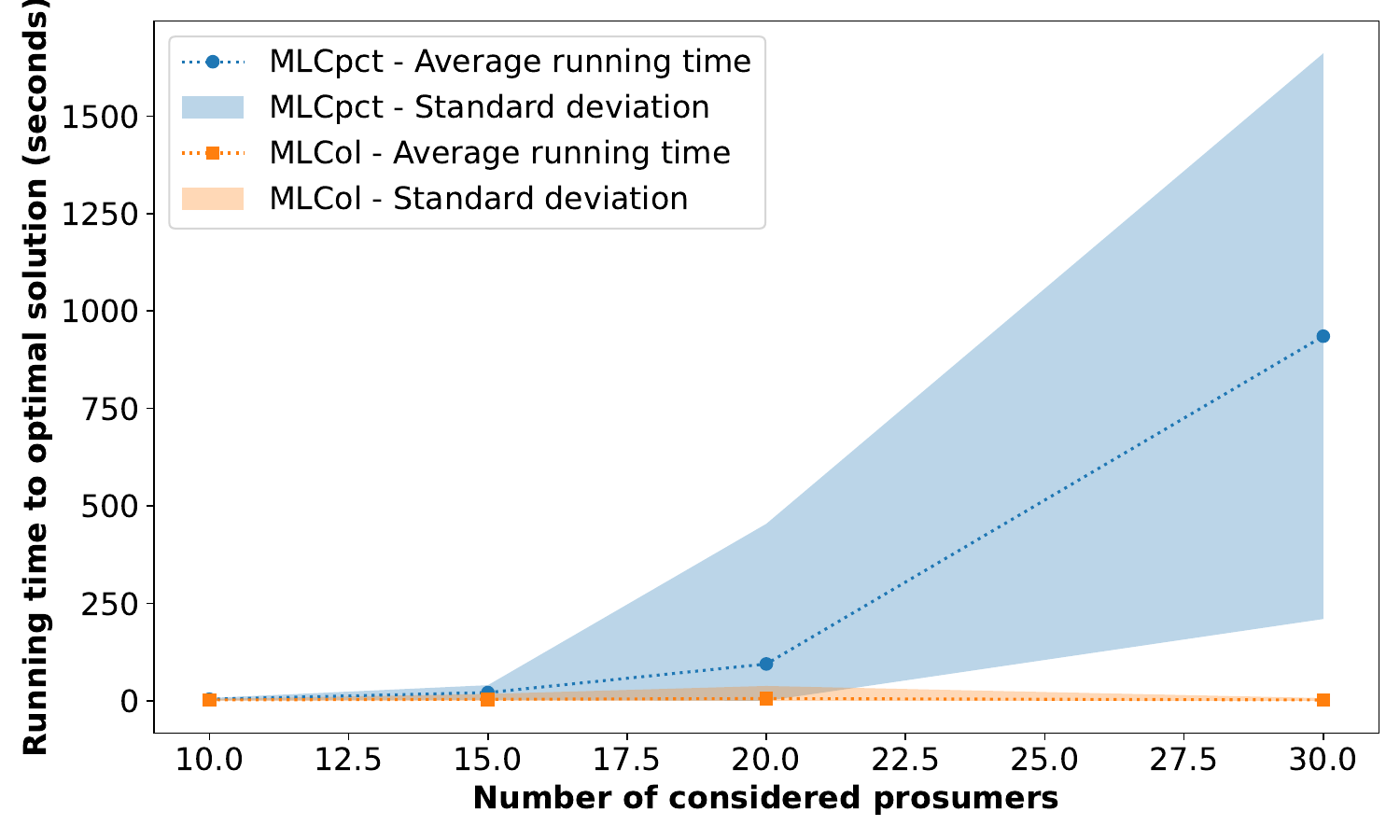}
        \caption{Across number of actors (7 days)}
        \label{fig:actors}
    \end{subfigure}
    
    \caption{Evolution of resolution running time of compact and decomposed models. \rev{SLCpct : Single loop compact model - SLExt : Benders extended single loop model - MLCpct : Multiloop compact model - MLCol : Generated cliques multiloop model}}
    \label{fig:runningtime-performances}
\end{figure}

As explained in the model presentation, extending the time horizon significantly increases the complexity of the problem, since each time step can influence whether or not an actor should be included in a collective self-consumption operation. 
Whether an actor belongs to or is excluded from a self-consumption operation is determined for the entire optimisation time horizon and cannot be changed from one time step to another. The existence of associated integer variables links each time step in the optimisation problem, and in the case of the compact formulation, requires all time steps to be solved within the same problem. The Benders extended reformulation allows for the extraction of integer variables and fixing them in the master problem, thereby enabling each time step of the sub-problem to be easily solvable independently. The increase in complexity over time horizon is the reason why a Benders decomposition of the model has been proposed. \cref{fig:horizon} illustrates the evolution of the resolution time of the single loop models for different time horizons, comparing the compact formulation with the extended formulation.

To assess variability, each scenario is represented with a set of 100 generated instances. This sampling covers variability in geographic distribution, prosumers' installed power, and consumption profiles. The compared computing time only considers the solving computing time and does not account for the time allocated to defining the problems (compact model, master problem, and sub-problems) or the time allocated to adding Benders cuts between the resolution of the master problem and sub-problems. When focusing on the standard deviation of each scenario, we observe significant variations in resolution time depending on consumption/production data and geographical distribution. However, for equivalent resolution time, the variability is fairly similar for extended and compact formulations.

It can be observed that a short time horizon is effectively handled by both compact and extended formulations. In many cases, the compact formulation performs even better than the extended formulation, as the single problem remains simple enough to be solved in less time than two separate problems. However, for larger time scales, the decomposed versions demonstrate better resolution times. This difference becomes more pronounced with larger time horizons, highlighting the benefits of using decomposed models. \cref{fig:horizon} illustrates resolution times on relatively small instances (10 actors). Tests have also been conducted on larger instances, showing a similar increase in resolution time for the compact model, while the extended model maintains reasonable resolution times. These observations highlight the advantages of using a decomposed formulation for larger time horizons.

The number of actors is a parameter that should influence the design aspect of the problem. The proposed clique generation algorithm has been applied to several instances with different numbers of prosumers. To keep the number of generated cliques reasonable, the selected instances feature only clustered spatial distributions of actors. This clustered distribution is generated by randomly sampling $p$ clusters across a square area, ensuring that the overall actor density is maintained. Clusters are generated to contain between 4 to 6 actors each, and the number of clusters is determined accordingly.

\cref{fig:actors} compares the computing time of the compact multiloop model (MLCpct) with the reformulation that integrates clique generation (MLCol). Similar to the evaluation of scalability in the temporal dimension, the computing time associated with the problem creation and addition of cuts are not considered in the comparison. The observations made previously regarding the time horizon dimension can equivalently be applied here for the number of actors. For a low number of actors, compact and reformulated models show similar performances. However, as the number of actors increases, the model involving prior clique generation outperforms the compact model. Tests conducted on larger instances (more than 30 actors) demonstrated that the compact formulation is unable to find the optimal solution in less than an hour, whereas clique generation drastically reduces the complexity of the problem and solving time.

Note that the method we propose through clique generation works well on realistic instances, specifically because their corresponding neighbourhood graph follows the sparsity assumption. We stress-tested the proposed formulation on complete graphs, observing a steep increase in memory requirements making the models intractable.

\rev{We emphasize the value of the proposed models, which provide exact and scalable solutions in both temporal and spatial dimensions. This is one of the main strengths of the proposed models compared with previous work, where the design problem was limited in terms of actors and planning horizon, or where selected actors were assumed to be fixed inputs.}

\subsection{Root-node relaxation tightness}

\begin{figure}[htbp]
  \centering

  \begin{subfigure}[b]{0.49\textwidth}
    \includegraphics[width=\textwidth]{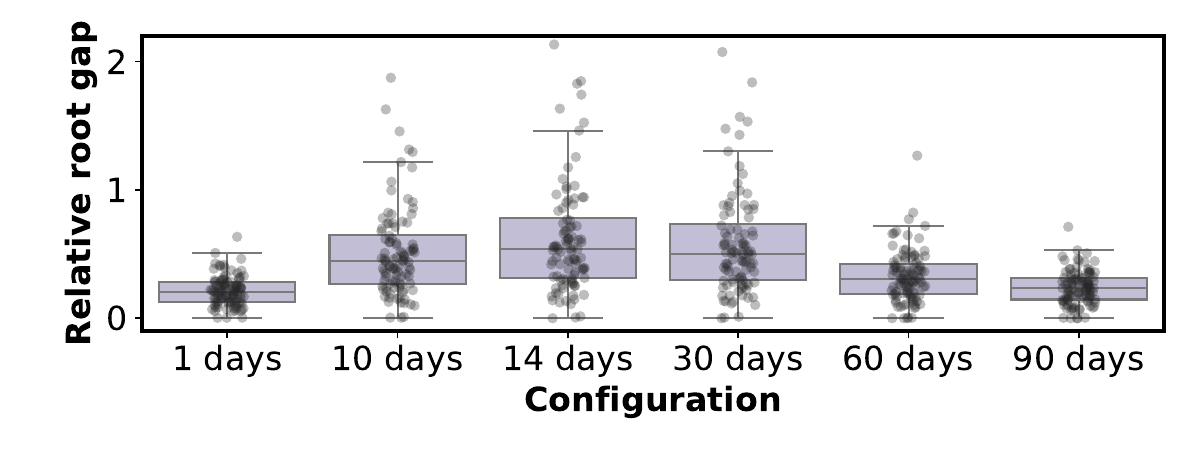}
    \caption{Across time horizon (10 actors) - SLCpct model. Three outliers removed for 10 days (value: 7.46), 14 days (value: 21.53) and 30 days (value: 21.50) configurations}
    \label{fig:root_gap_sl_days}
  \end{subfigure}
  \begin{subfigure}[b]{0.49\textwidth}
    \includegraphics[width=\textwidth]{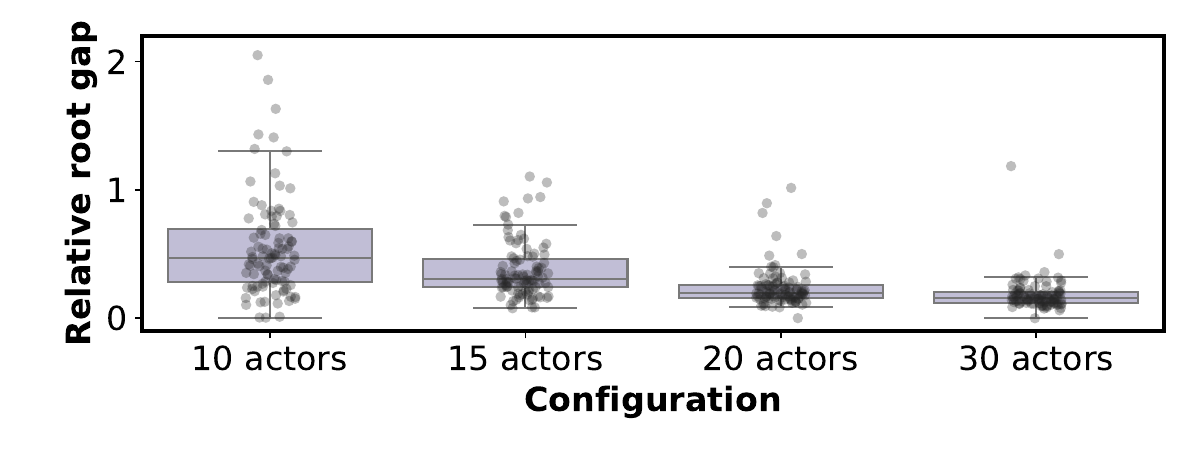}
    \caption{Across number of actors (7 days) - SLCpct model. One outlier removed for 10 actors configuration (value: 11.20)\\~}
    \label{fig:root_gap_sl_actors}
  \end{subfigure}

  \begin{subfigure}[b]{0.49\textwidth}
    \includegraphics[width=\textwidth]{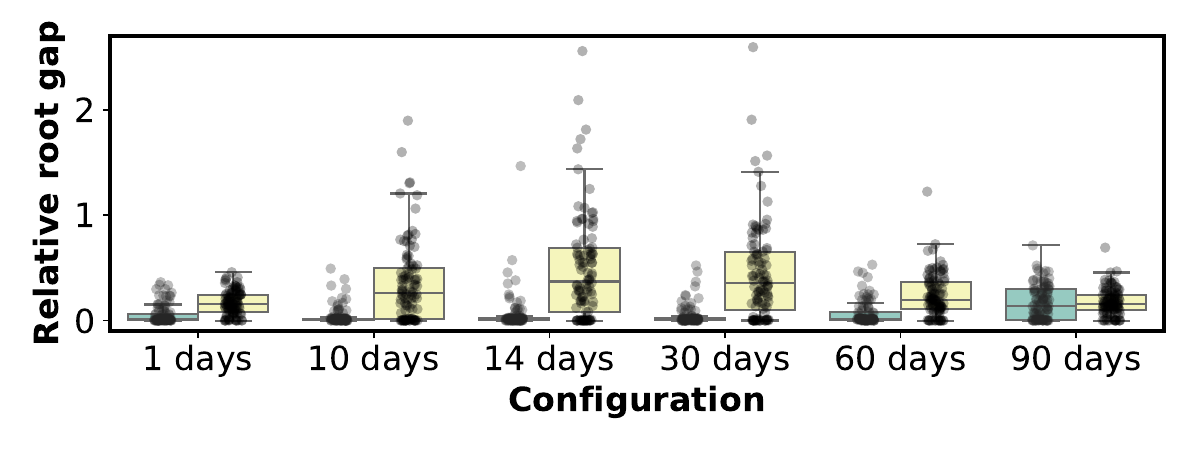}
    \caption{Across time horizon (10 actors) - MLCpct (left) and MLCol (right) models. One outlier removed for 30 days configuration (MLCol - value: 17.10)\\~}
    \label{fig:root_gap_ml_days}
  \end{subfigure}
  \begin{subfigure}[b]{0.49\textwidth}
    \includegraphics[width=\textwidth]{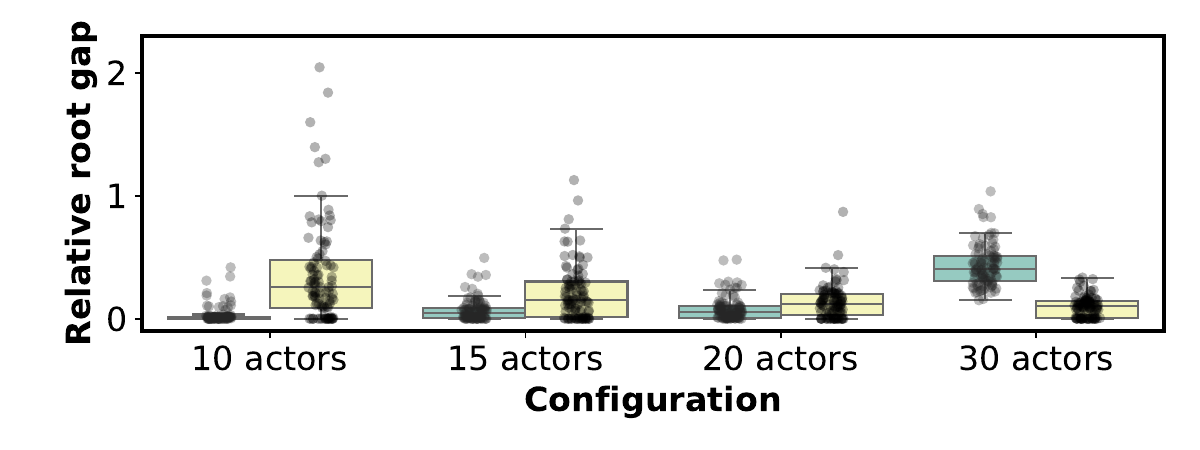}
    \caption{Across number of actors (7 days) - MLCpct (left) and MLCol (right) models. Two outliers removed at 10 actors (MLCol - value: 10.27) and 30 actors configuration (MLCpct - value: 3.64)}
    \label{fig:root_gap_ml_actors}
  \end{subfigure}

  \caption{Root gap distribution by solving method across configurations (some outliers have been removed from the plots for clarity purpose)}
  \label{fig:root_gap}
\end{figure}

We analyse in this section the quality of the formulation through the root-node gap which provides a measure of how tight the different relaxations are. \rev{The root-node gap is defined as the relative difference between the optimal value of the linear programming relaxation and the value of the integer formulation before branching.}
\cref{fig:root_gap} illustrates how the root-node gap evolves with respect to the time horizon and the number of considered prosumers for the different problem formulations.
Focusing first on single loop models, \cref{fig:root_gap_sl_actors} shows that increasing the number of prosumers results in smaller relative gaps, thus with linear relaxations that are tighter around integer solutions. Interestingly, as illustrated in \cref{fig:root_gap_sl_days}, the root-node gap is maximal for ten prosumers between 14 and 30 days of horizon, with shorter and longer horizons both resulting in smaller gaps and tight formulations, which leads us to conjecture that longer planning horizons increases the contrast in the economic interest to select actors, reducing the root-node fractionality.

Having a lower relative root gap however, does not imply that the problem can be solved in a shorter time, as shows \cref{fig:horizon} for the single loop formulation. Nevertheless, it indicates that the resolution should be easier relatively to the size of the mixed-integer formulation (number of variables and constraints).

This tendency does not hold for multiple-loop compact formulations. In these cases, the root gap tends to increase with longer time horizons (greater than 30 days). The same effect is particularly true when increasing the number of prosumers, especially from 20 to 30 prosumers, where a significant root gap difference is observed.

What remains noteworthy, though, is that the reformulated problem using clique enumeration exhibits the opposite behaviour. Similar to single loop models, it exhibits a peak root gap around 14 days of time horizon and decreases for larger time horizons.

\cref{fig:root_gap_ml_days} and \cref{fig:root_gap_ml_actors} highlight the benefits of the clique generation formulation. This approach results in lower root gaps than the compact model once the number of prosumers becomes large (i.e., greater than 20). Furthermore, unlike the compact formulation, clique generation not only produces smaller root-node gaps but also tends to further reduce them as either the time horizon or the number of prosumers increases, hinting that our approaches will produce high-quality solutions and guarantees even without running a full branch-and-cut algorithm.

\subsection{Solution sensitivity analysis}
In addition to runtime performance, another important focus is on the solutions derived from the multiloop model from the application perspective. This entails evaluating the solutions in terms of their cost effectiveness and assessing the influence of operational parameters on proposed solutions.

\subsubsection{Influence of the legal maximum installed power}

One criterion that constrains the formation of collective self-consumption loops is the legal maximum authorised installed power per loop. Sensitivity to this parameter can be important to the operator defining the loops, since this legal requirement may change with legislation.

\begin{figure}[h]
    \centering
    \begin{subfigure}[b]{0.33\textwidth}
        \centering
        \includegraphics[width=\textwidth]{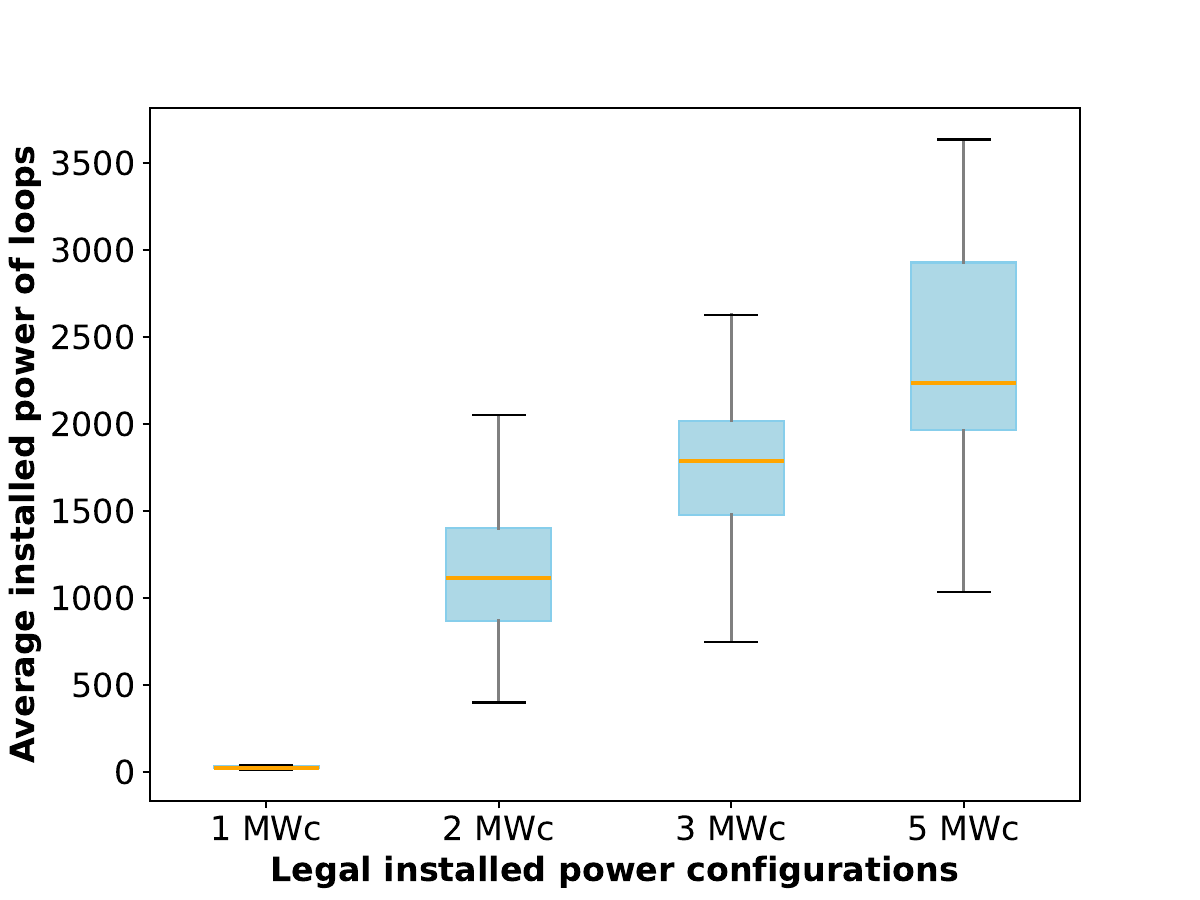}
        \caption{Installed power of the loops}
        \label{fig:power_looppower}
    \end{subfigure}
    \begin{subfigure}[b]{0.33\textwidth}
        \centering
        \includegraphics[width=\textwidth]{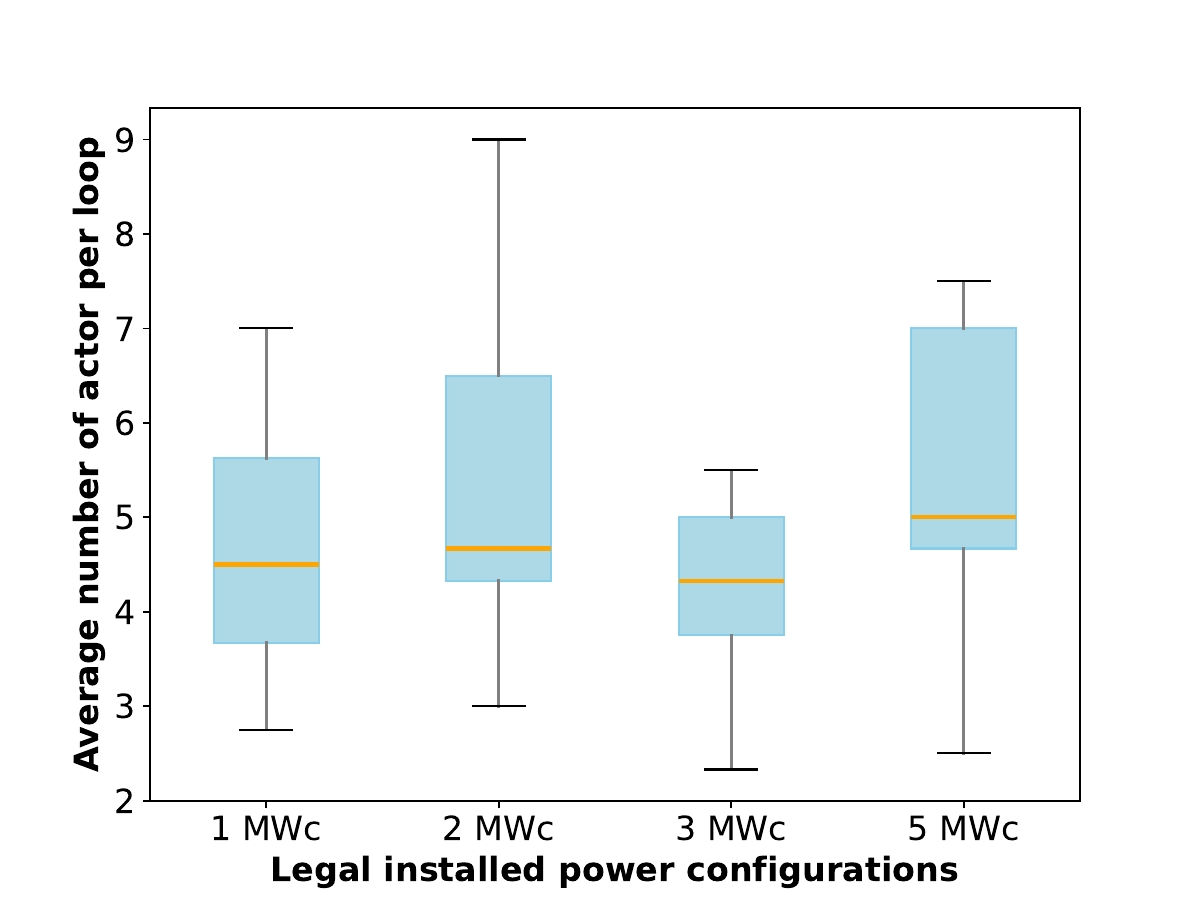}
        \caption{Size of the loops}
        \label{fig:power_nact}
    \end{subfigure}
    \begin{subfigure}[b]{0.33\textwidth}
        \centering
        \includegraphics[width=\textwidth]{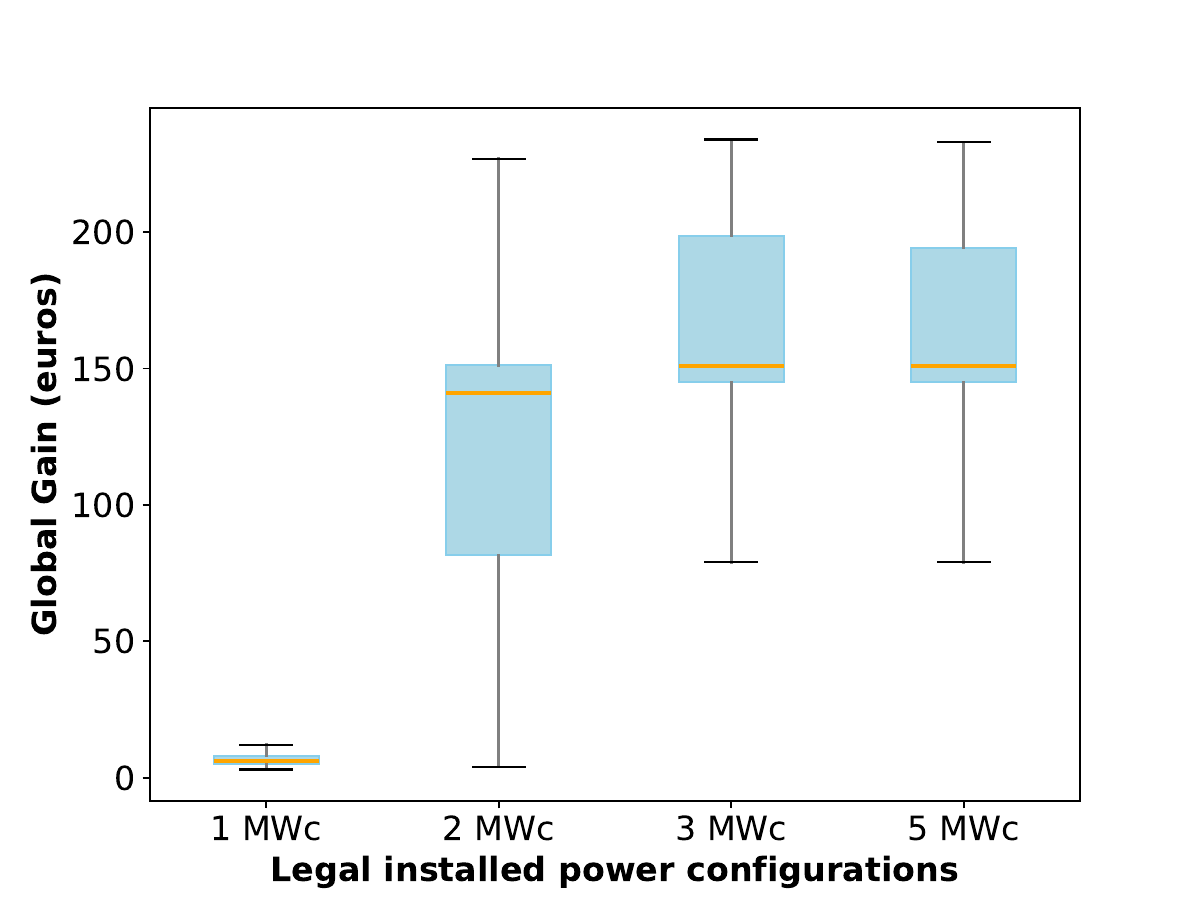}
        \caption{Global gain of the loops}
        \label{fig:power_gain}
    \end{subfigure}
    \begin{subfigure}[b]{0.33\textwidth}
        \centering
        \includegraphics[width=\textwidth]{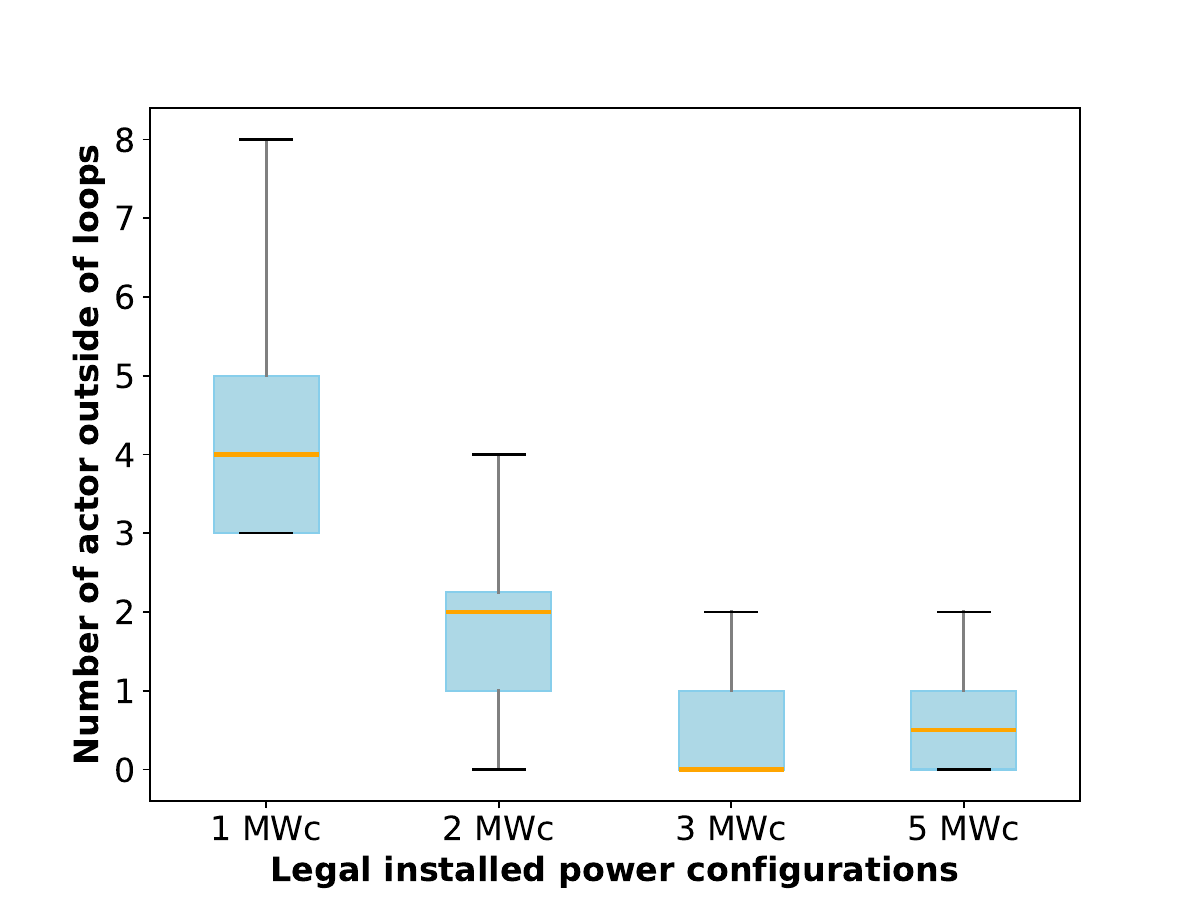}
        \caption{Number of actors not in a loop}
        \label{fig:power_actwo}
    \end{subfigure}
    \caption{Influence of the legal power on operational indicators}
    \label{fig:power-all}
\end{figure}

\cref{fig:power_looppower} and \ref{fig:power_gain} respectively illustrate the impact of this parameter on the average size of the corresponding loops (in terms of installed power) and the benefit of the formed collective self-consumption loops.

Unsurprisingly, the average installed power of created loops grows with the value of $P_{\mathrm{leg}}$, allowing the loops to reach larger installed power, which results in the model attempting to maximise the size of the loops. However, \cref{fig:power_nact} shows that the number of actors per loop is not significantly affected by the different $P_{\mathrm{leg}}$ configurations. This can be explained by the fact that accepting a new large producer in the loop may result in the removal of another actor, based on distance criteria for example.

Instances having an average installed power per loop that remains low even though not restricted by $P_{\mathrm{leg}}$ can be explained by cases where all the large producer groups are concentrated in a single cluster. This results in a loop with one or two big producers, while other loops only contain actors with low installed power. This clustering effect can lead to instances where the distribution of actors across loops is uneven, causing some loops to have significantly higher installed power than others.
One also observes the creation of loops with very small installed power when $P_{\mathrm{leg}}$ is set to 1 MWc. This is primarily caused by the gap in installed power between 9 kWc and 1 MWc in the instances used. In this configuration, no created loop can contain any of the professional producers, resulting in loops only containing small producers and consequently having low installed power.

These observations, made on relatively small instances (15 actors), demonstrate how an optimisation tool based on the proposed model can be used to evaluate the viability of different types of production plants. By examining the impact of parameters such as the legal maximum authorised installed power ($P_{\mathrm{leg}}$), decision-makers can gain insights into the optimal configuration.

\begin{figure}[h]
    \centering
    \begin{subfigure}[b]{0.33\textwidth}
        \centering
        \includegraphics[width=\textwidth]{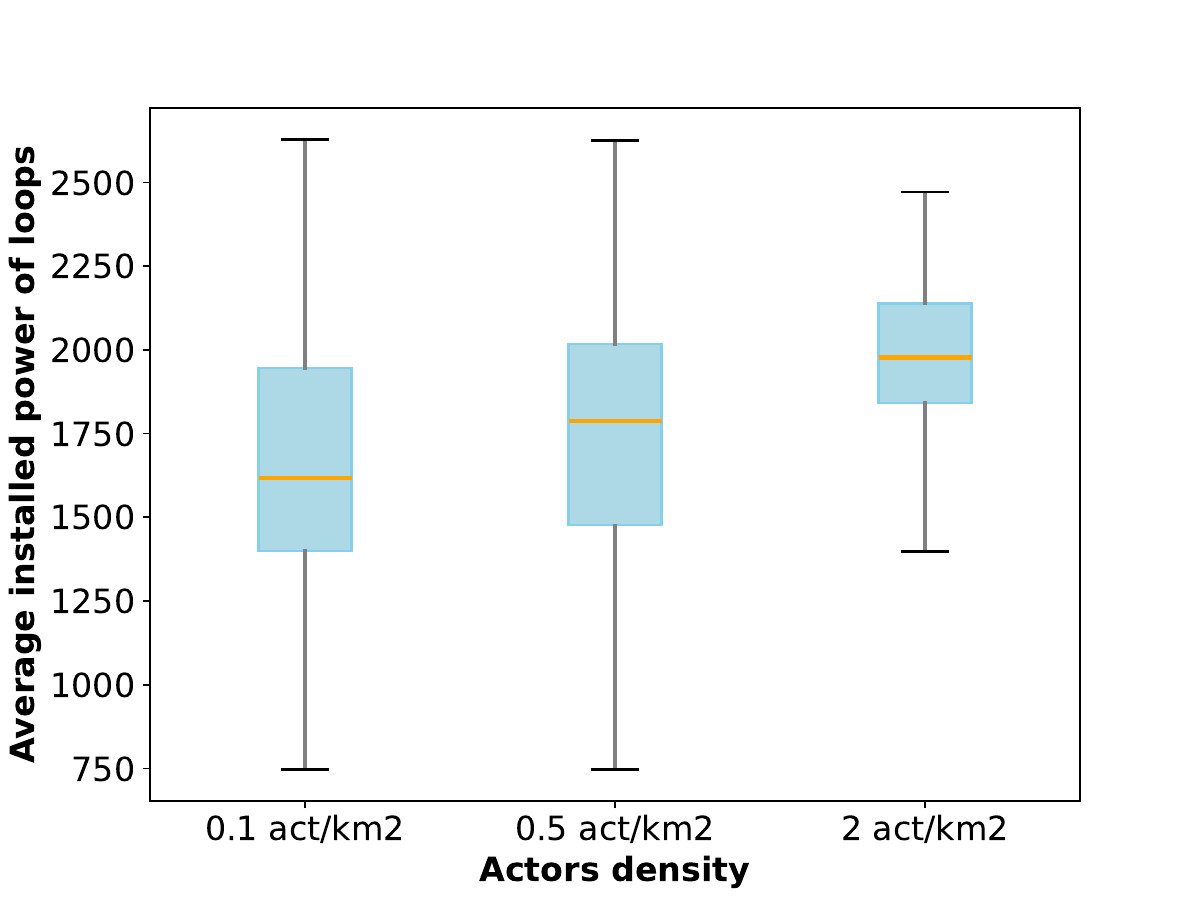}
        \caption{Installed power of the loops}
        \label{fig:density_looppower}
    \end{subfigure}
    \begin{subfigure}[b]{0.33\textwidth}
        \centering
        \includegraphics[width=\textwidth]{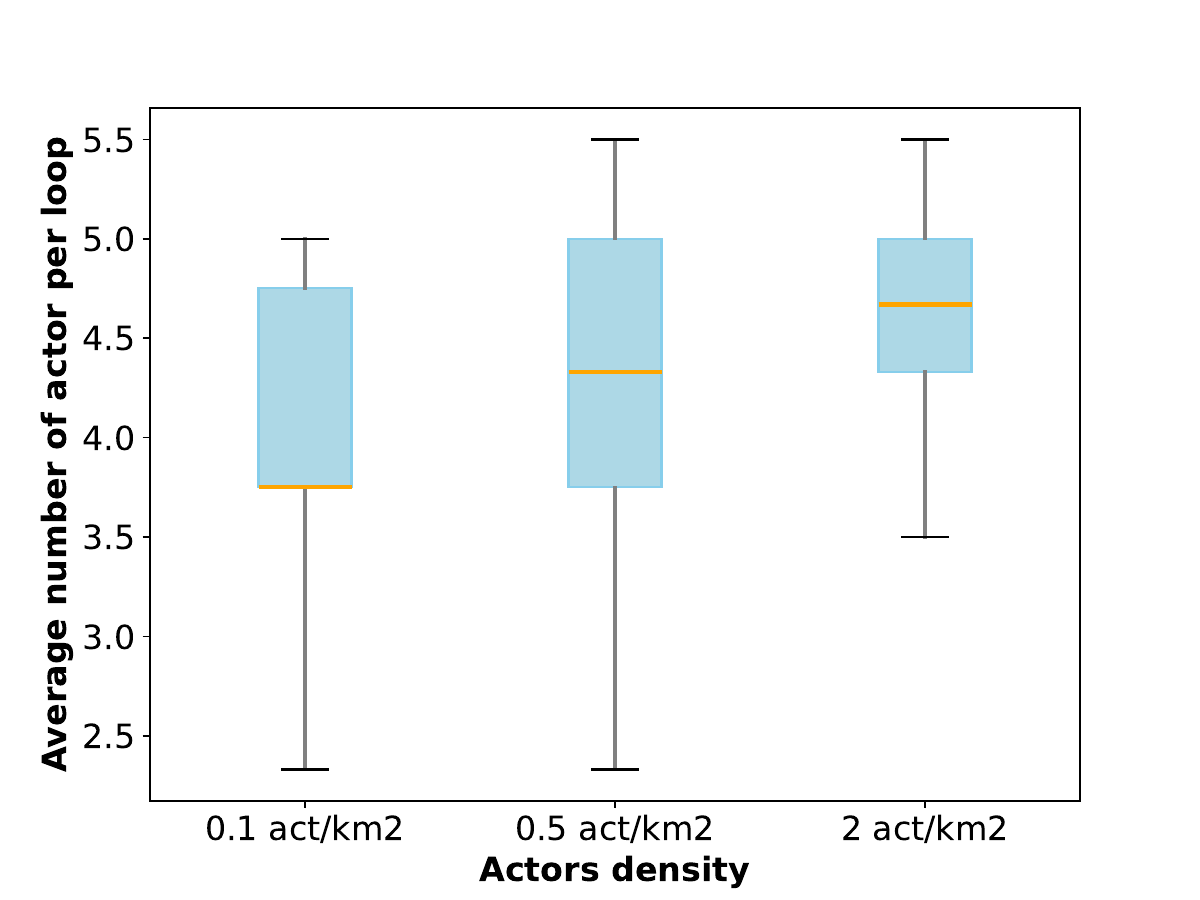}
        \caption{Size of the loops}
        \label{fig:density_nact}
    \end{subfigure}
    \begin{subfigure}[b]{0.33\textwidth}
        \centering
        \includegraphics[width=\textwidth]{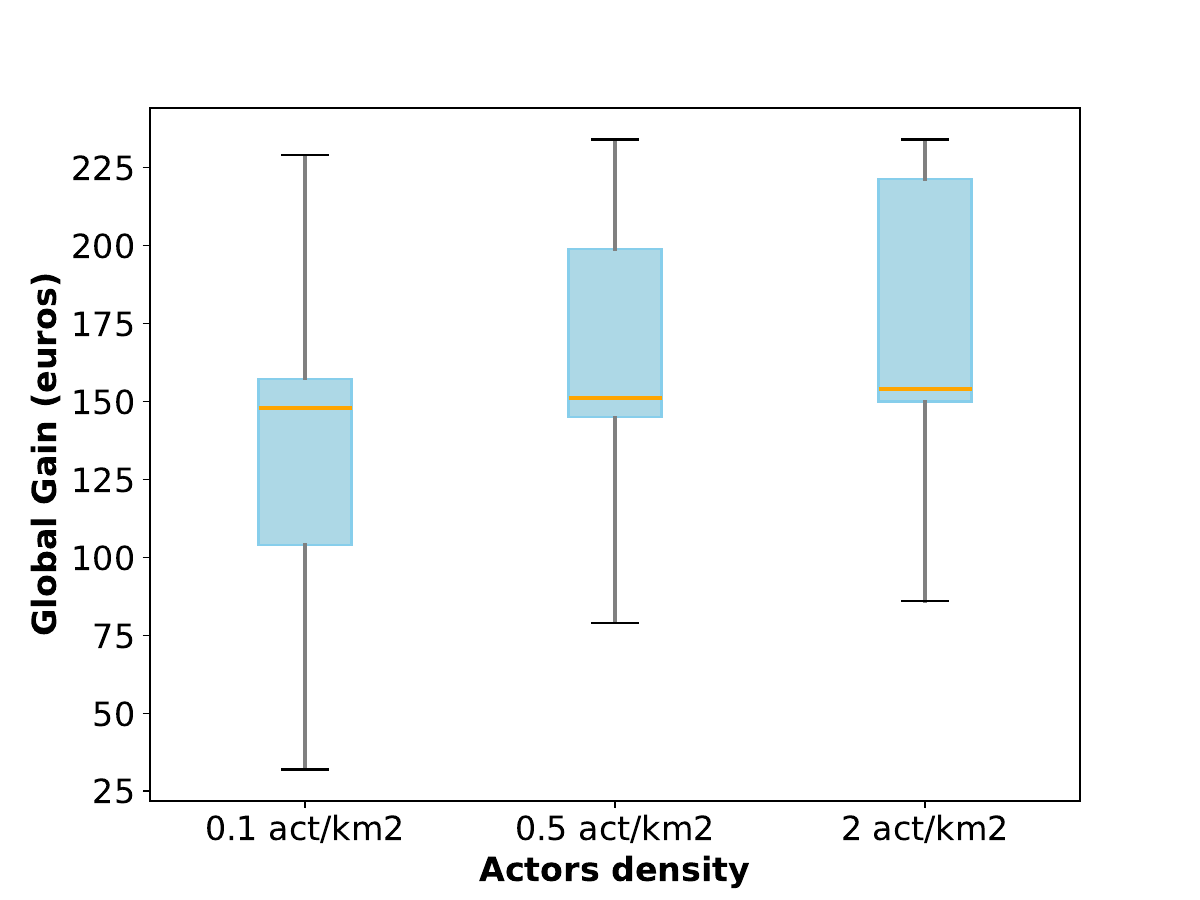}
        \caption{Global gain of the loops}
        \label{fig:density_gain}
    \end{subfigure}
    \begin{subfigure}[b]{0.33\textwidth}
        \centering
        \includegraphics[width=\textwidth]{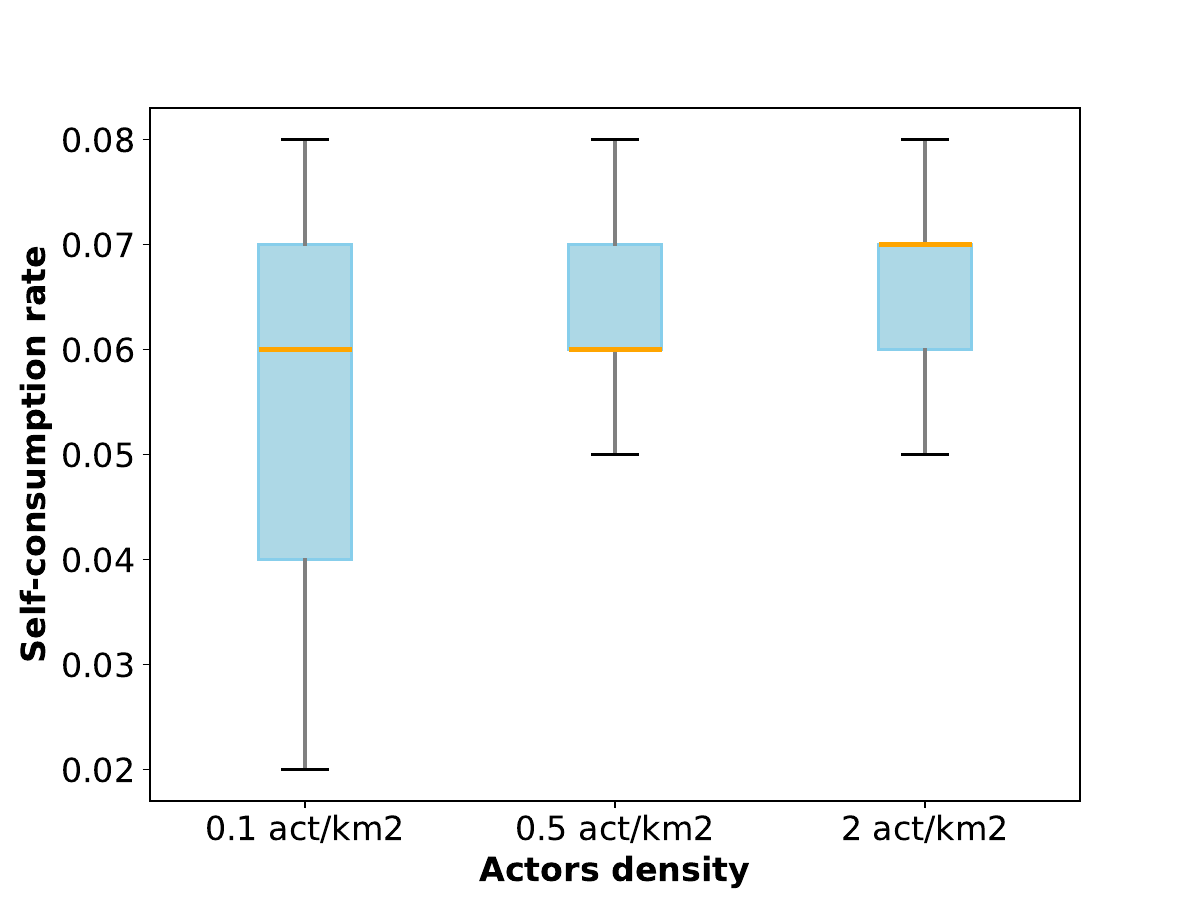}
        \caption{Self-consumption rate}
        \label{fig:density_scr}
    \end{subfigure}
    \caption{Influence of distribution density on operational indicators}
    \label{fig:density-all}
\end{figure}

\subsubsection{Influence of actor density}

Increasing the actor density favours the formation of larger collective self-consumption communities. \cref{fig:density_looppower} illustrates a slight increase in the installed power of loops as clusters are brought closer together. This reflects the fact that separated clusters have a greater chance of being at a reasonable distance in denser configurations. This observation is confirmed by analyzing the average number of actors per loop in \cref{fig:density_nact}. Since larger loops ensure a greater amount of locally produced electricity, the overall gain realised is also directly affected by an increase in density, as shown in \cref{fig:density_gain}.

However, this increase in gain can be nuanced when considering the self-consumption rate. The increase in self-consumption rate is not as pronounced as for other indicators. This is because the self-consumption rate is influenced not only by the amount of locally produced electricity but also by the ability to consume this electricity. From \cref{fig:density_scr}, we can assume that most configurations, even though they allow for larger loops as density increases from $\text{0.5 act/km}^2$ to $\text{2 act/km}^2$, do not necessarily provide the corresponding consumption profiles to increase the self-consumption rate. However, this indicator could potentially be improved through the inclusion of more net consumers, which would raise the consumption potential.

In real-world applications, the density of actor distribution is not an actionable parameter but rather a result of the choice of considered candidates for self-consumption communities. Increasing the number of candidates should therefore be advantageous for profitable self-consumption communities, without necessarily guaranteeing an increase in collective self-consumption rates. 

\begin{figure}
    
\end{figure}

\begin{figure}[h]
    \centering
    \begin{subfigure}[b]{0.49\textwidth}
        \centering
        \includegraphics[width=1\linewidth]{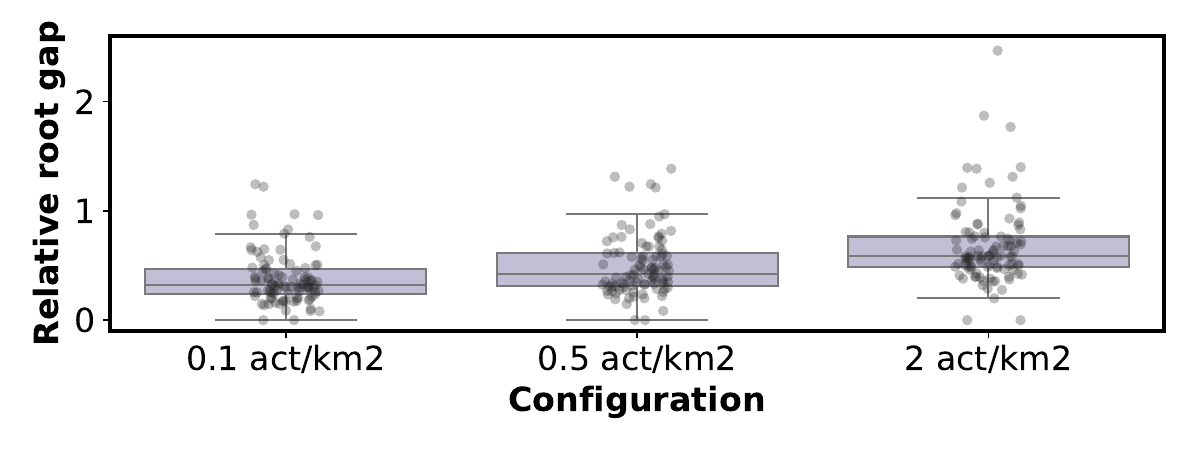}
        \caption{Root gap evolution along actors density for SLCpct model}
        \label{fig:root_gap_density_sl}
    \end{subfigure}
    \begin{subfigure}[b]{0.49\textwidth}
        \centering
        \includegraphics[width=1\linewidth]{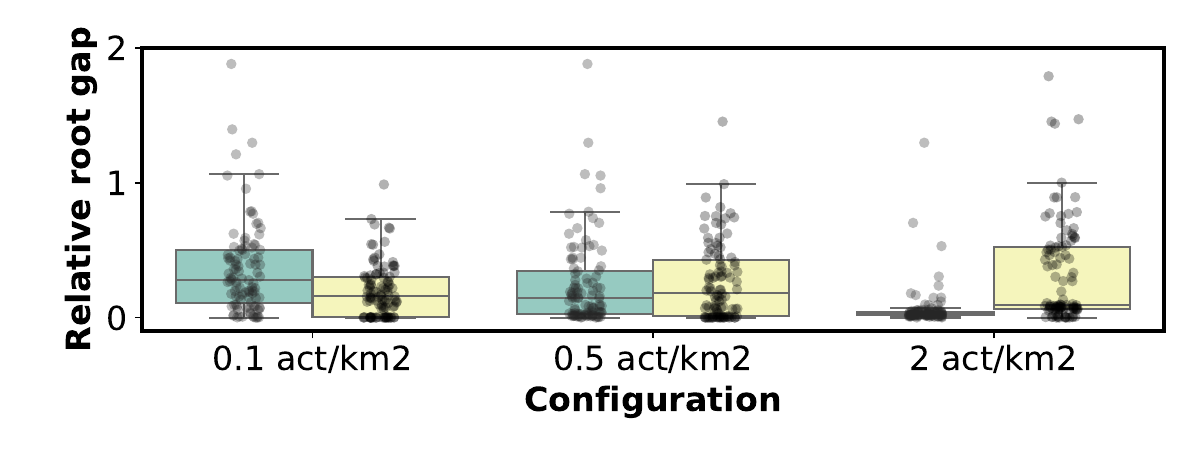}
        \caption{Root gap evolution along actors density for MLCpct (left) and MLCol (right) models}
        \label{fig:root_gap_density_ml}
    \end{subfigure}
    
    \caption{Root gap distribution by solving method along actors density}
    \label{fig:root_gap_density}
\end{figure}

\cref{fig:root_gap_density} shows the relative root gap for group of instances having different prosumers density. We observe that the compact formulation performs better for denser configurations, as having close actors is similar to relaxing the distance constraint, leading to an actors selection problem driven only by the benefit of loops and constrained by power constraint. This figure also shows that the formulation involving clique generation (MLCol) is not suited for denser problems, as having closer actors conducts to a high number of possible cliques, and a bigger number of variables. For sparse configurations however the clique generation performs better than the compact formulation, as many invalid actors combinations are not evaluated.

\subsection{Influence of exposition and period of the year}

\begin{figure}[h]
    \centering
    \begin{subfigure}[b]{0.328\textwidth}
        \centering
        \includegraphics[width=\textwidth]{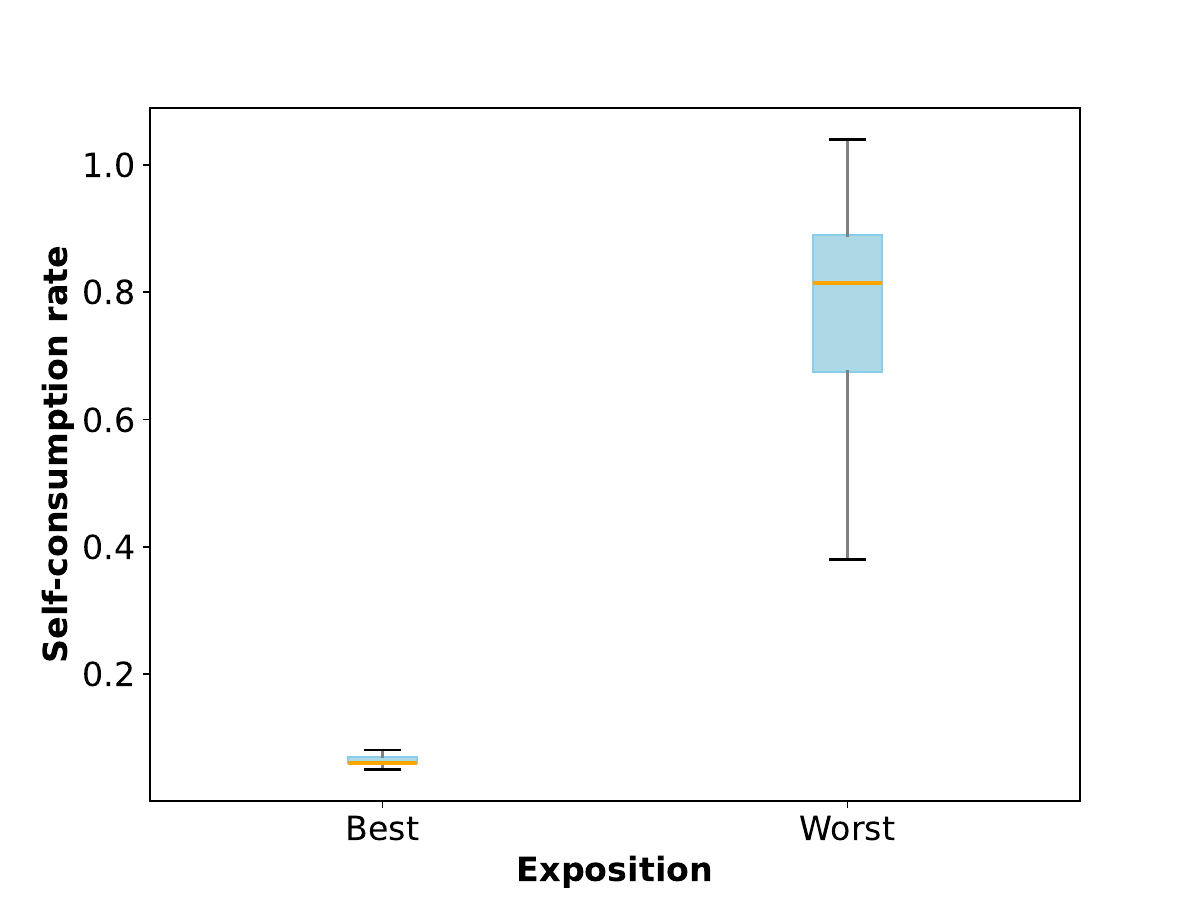}
        \caption{Self-consumption rate}
        \label{fig:exposition_scr}
    \end{subfigure}
    \begin{subfigure}[b]{0.328\textwidth}
        \centering
        \includegraphics[width=\textwidth]{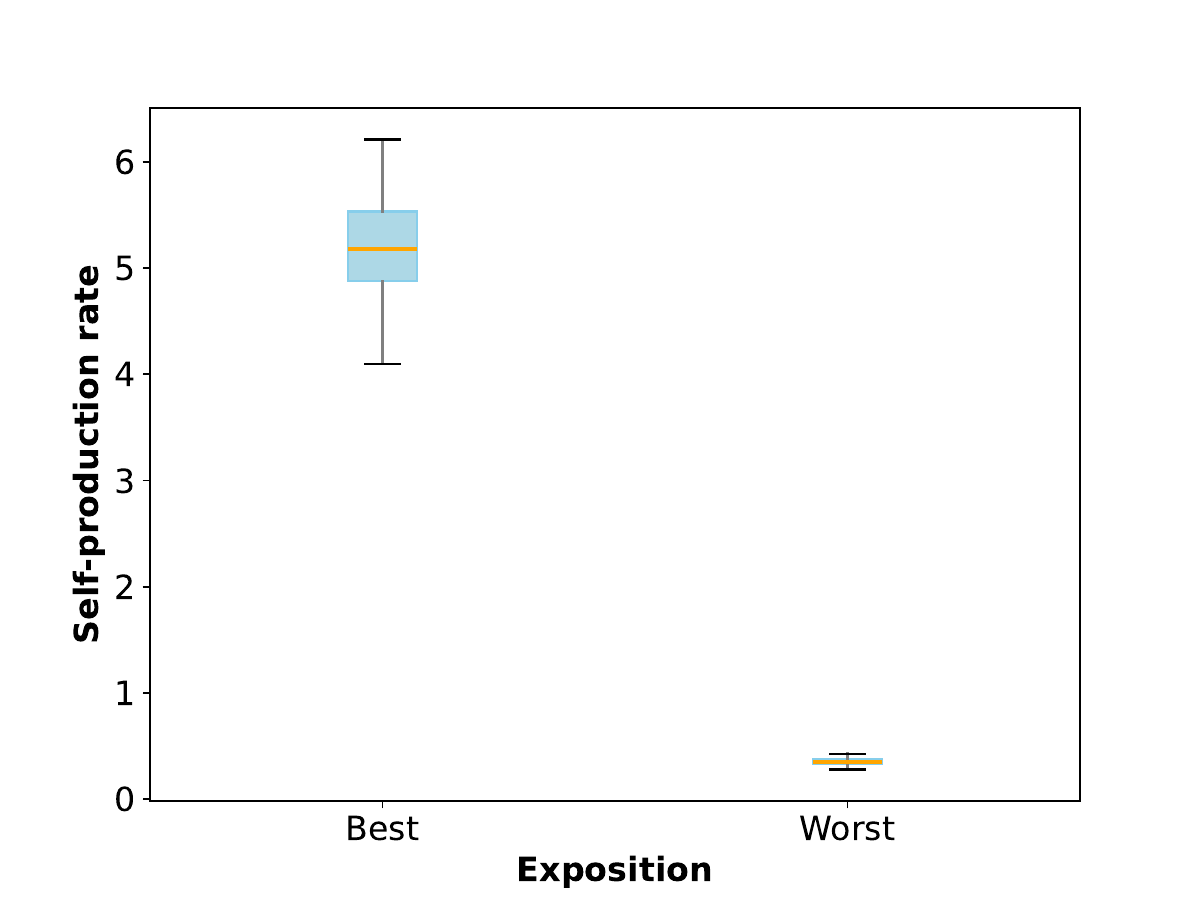}
        \caption{Covering rate}
        \label{fig:exposition_spr}
    \end{subfigure}
    \begin{subfigure}[b]{0.328\textwidth}
        \centering
        \includegraphics[width=\textwidth]{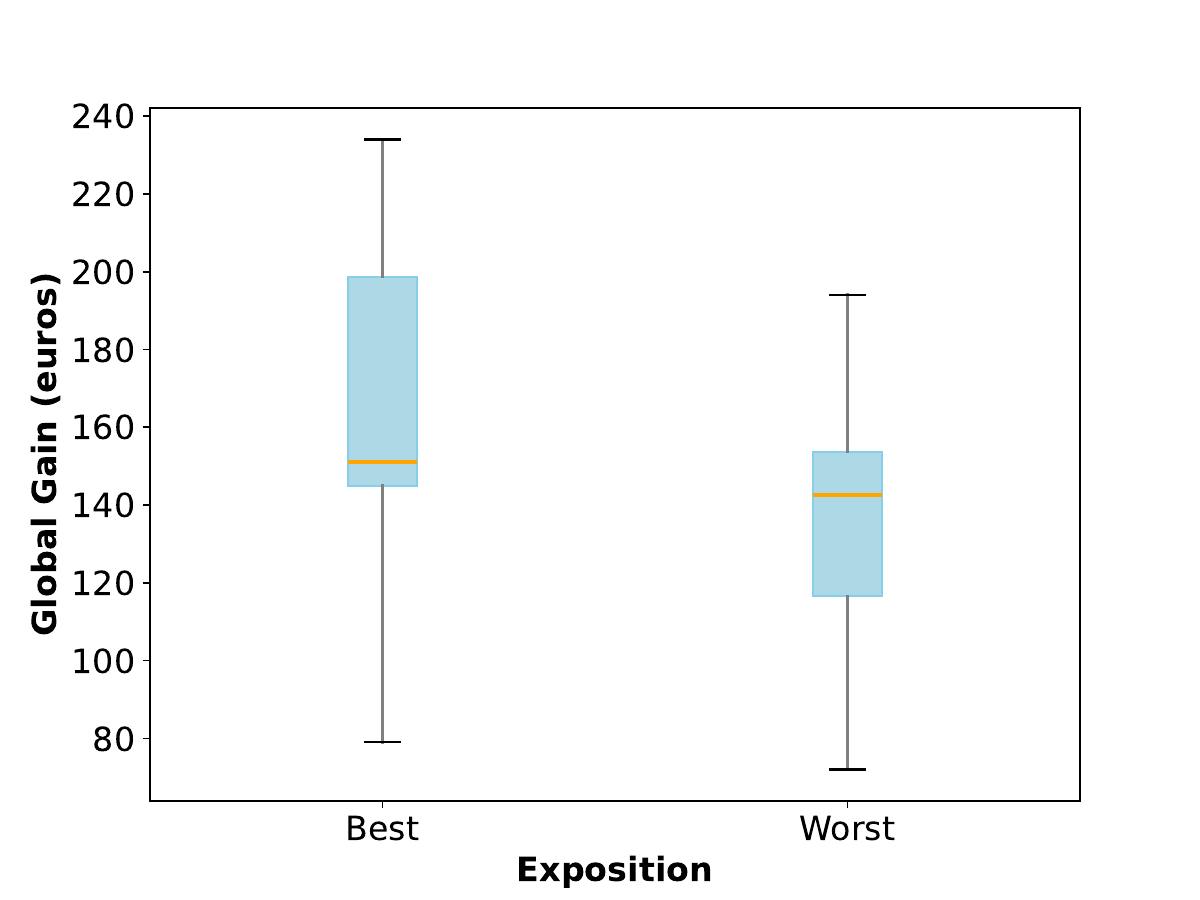}
        \caption{Global gain of the loops}
        \label{fig:exposition_gain}
    \end{subfigure}
    \caption{Influence of exposition on operational indicators}
    \label{fig:exp-all}
\end{figure}

Solar panel exposure directly impacts production volumes as it influences the production yield of plants. Even though the best configurations logically favour greater economic benefits (see \cref{fig:exposition_gain}), dealing with different expositions may result in changes in operational indicators. In the worst exposition case, the total production amount is lower than consumption needs, whereas it is higher with the best exposition configuration. This results in an inversion in terms of self-production rate and self-consumption rate, as \cref{fig:exposition_scr} and \ref{fig:exposition_spr} demonstrate.

This highlights the fact that the self-consumption indicator is not always in line with the economic objective, which is computed relative to the production volumes. Guiding the optimisation through the economic objective is therefore a way to balance this effect, taking into account both self-consumption and self-production indicators. By optimising based on economic objectives and under the assumption that surplus selling prices are lower than electricity costs, decision-makers can ensure that decisions align with both financial goals and sustainability targets.

\begin{figure}[h]
    \centering
    \begin{subfigure}[b]{0.328\textwidth}
        \centering
        \includegraphics[width=\textwidth]{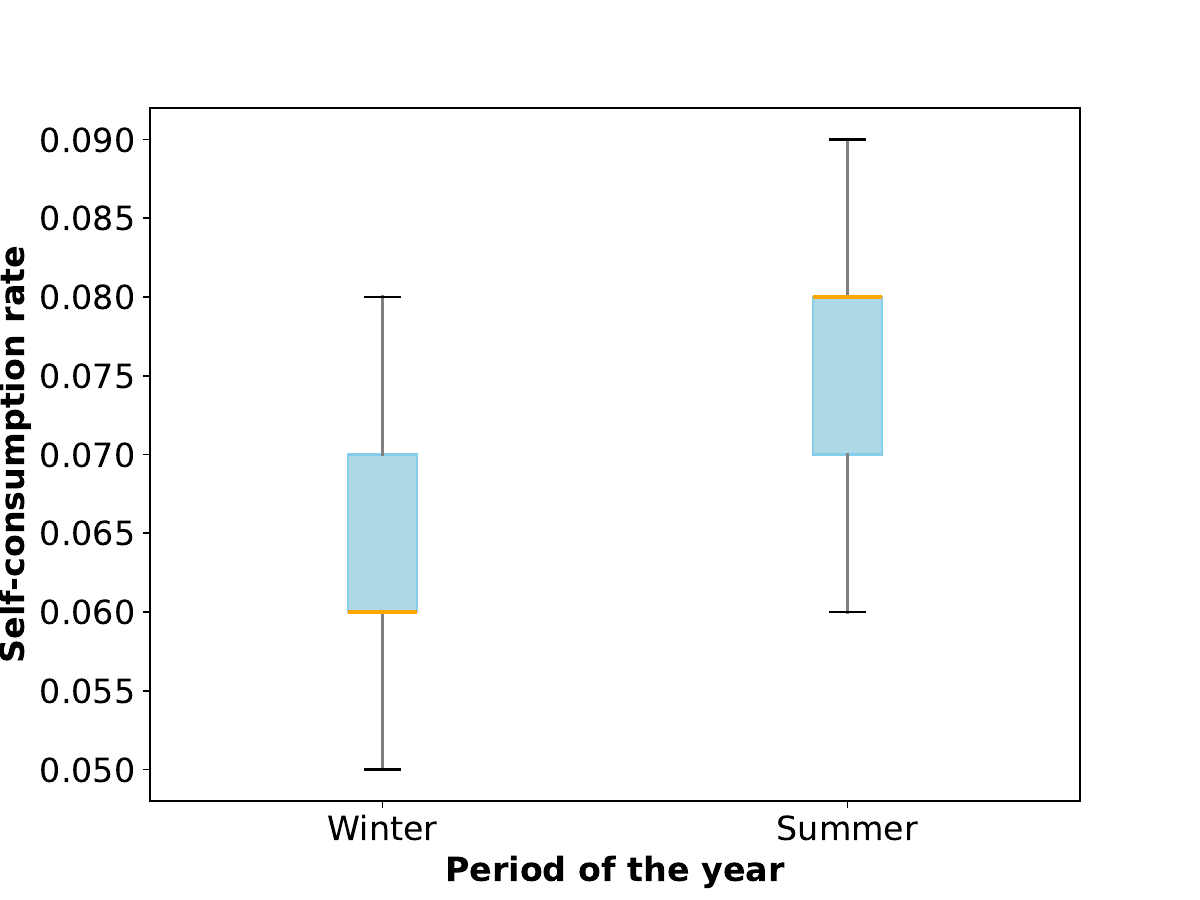}
        \caption{Self-consumption rate}
        \label{fig:year_scr}
    \end{subfigure}
    \begin{subfigure}[b]{0.328\textwidth}
        \centering
        \includegraphics[width=\textwidth]{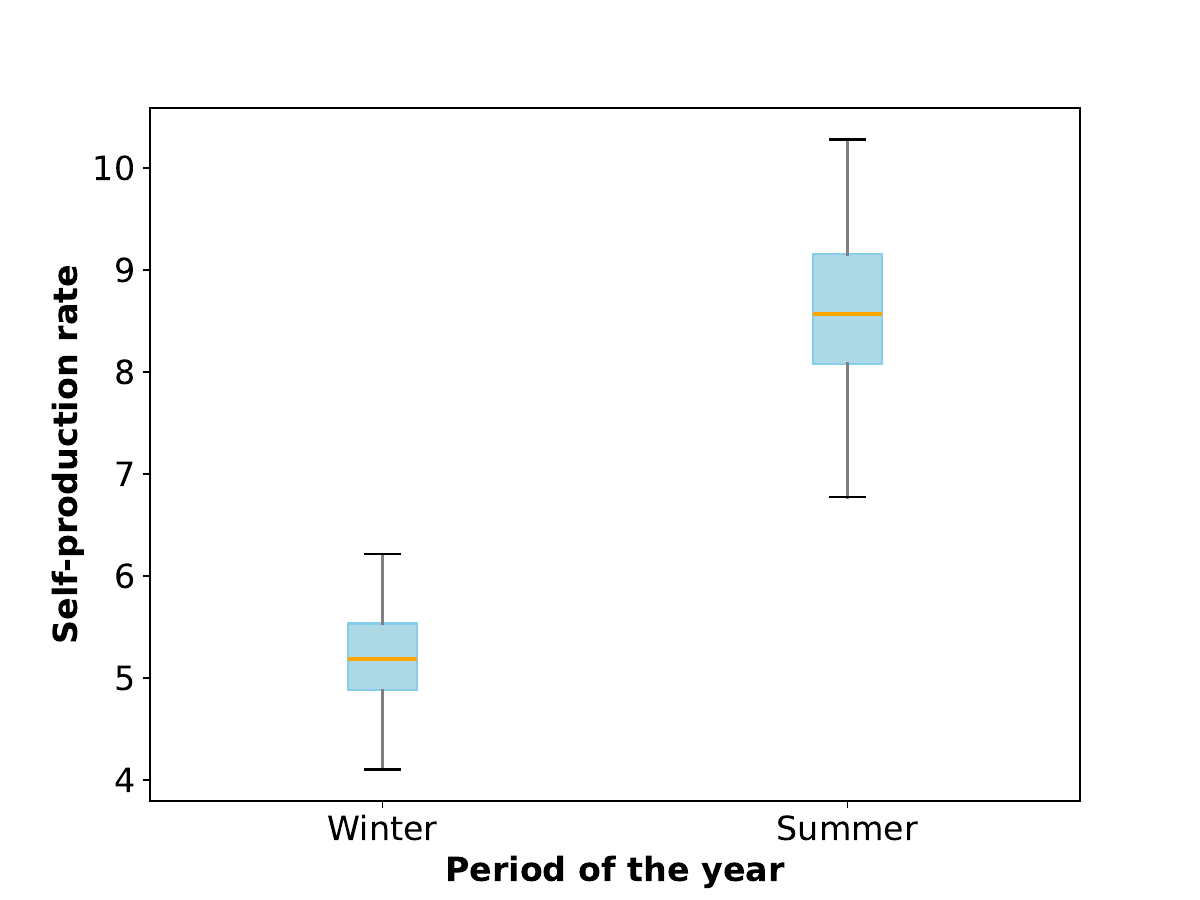}
        \caption{Self-production rate}
        \label{fig:year_spr}
    \end{subfigure}
    \begin{subfigure}[b]{0.328\textwidth}
        \centering
        \includegraphics[width=\textwidth]{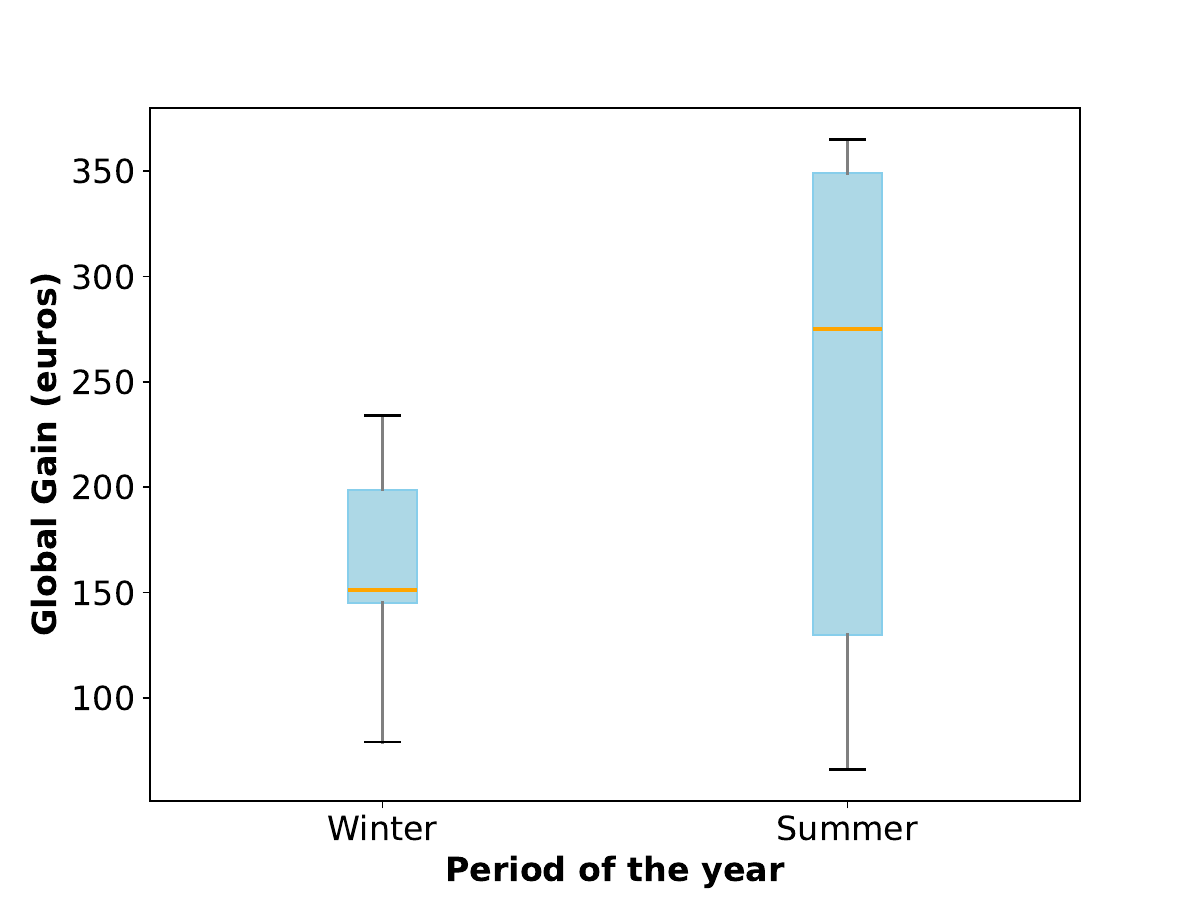}
        \caption{Global gain of the loops}
        \label{fig:year_gain}
    \end{subfigure}
    \caption{Influence of the period of the year on operational indicators}
    \label{fig:period-all}
\end{figure}

The previously mentioned inversion is not observed when the changes in production come from considering two different periods of the year (see \cref{fig:year_scr} and \ref{fig:year_spr}). This can be explained by the fact that the period of the year has less impact on production than the panel exposition. When production can be divided by 10 due to poor exposition, the ratio between summer and winter is less than 2. In the presented instances, this means that in both summer and winter periods, production volume exceeds consumption volume. Still, considering the summer period seems to be more in favour of self-production than self-consumption.

Realised gain is also more subject to variation during the summer period (see \cref{fig:year_gain}), as the producers with the largest installed power and their geographic placement have a greater influence during the summer than during winter. This makes the positioning and capacity of larger producers more critical for maximizing overall gain when considering summer production peaks.

\section{Conclusion and future work}
\label{sec:conclusion}
In this article, we presented various mixed-integer linear programming models to optimise the design of collective self-consumption communities. We first proposed models for the optimisation of the design addressing one (SLCpct) or multiple (MLCpct) loop problems. An interest of the presented models lies in their ability to address both the production flow distribution problem and the design aspect in a single model, as the design aspect of the model is rarely tackled in related work. To manage the complexity arising from this coupling, we proposed decomposition methods, demonstrating their effectiveness through tests on different types of instances. To cope with spatially growing multiple communities problems (inducing a higher number of prosumers) we proposed a Dantzig-Wolfe decomposed model (MLCol) and a dedicated algorithm for the enumeration of feasible loops leveraging the spatial sparsity of actors, avoiding the use of a branch-and-price algorithm. To make the models actionable on large time horizons - which is needed in the design task - we proposed a Benders decomposition of SLCpct and MLCol models leading to models that keep reasonable solving time on large time horizons (SLExt and MLColExt).

Perspective for future work include an enhanced version of the Benders' algorithm to select specific cuts and accelerate the solving process, and to determine a fair distribution of the benefits generated by the energy exchanges, providing decision-makers with a pricing policy. A third identified perspective is the addition of storage facilities in the models, which however breaks the independence of the time steps and opens computational challenges.

We aimed through this article to provide comprehensive insights into the optimisation of collective self-consumption communities, considering both technical and operational aspects. We conducted a performance-oriented analysis on instances of various sizes and configurations, and also performed a sensitivity analysis on operational parameters.
Our computational study shows how these models can support decision-making and help guide the influence of solar plant installation on self-consumption loop design.

\section{Acknowledgment}
We thank Amarenco for providing us with their self-consumption business model and fine-grained data which offered a basis for this work and a motivating use case.

\section{Compliance with Ethical Standards}
\subsection{Funding}
This study was funded by the French Ministry of Economy, Finance and Industrial and Digital Sovereignty through a \textit{France Relance} project.
Mathieu Besançon was supported by the French ANR through the MIAI Cluster Grenoble (reference ANR-23-IACL-0006).

\subsection{Ethical approval}
This article does not contain any studies with human participants or animals performed by any of the authors.


\begin{footnotesize}
\bibliographystyle{apalike} 
\bibliography{output}
\end{footnotesize}






\newpage
\appendix
\section{Summary of notations}\label{app:notations}
\begin{table}[h!]
\centering
\input{tables/notations}

\caption{Summary of notations (indices $s$ and $t$ that are used in some models are dropped for conciseness)}
\end{table}

\section{Dual model of the single loop Benders decomposition} \label{app:dual_benders}
The dual of Model~\eqref{model:continuousprojection} is given by the Model~\eqref{model:dualcontinuousprojection} \rev{below defined through Equation}~\eqref{eq:dualobjective1} \rev{to} ~\eqref{cons:dualbounds} :
\begin{subequations}\label{model:dualcontinuousprojection}
\begin{align}
\max_{\alpha,\pi,\gamma,\theta,\kappa,\nu,\lambda,\tau} \, & \sum_{i,j\in\Edgeset} (\alpha_{ij}^p x_i P_i + \alpha_{ij}^c x_j C_j) + &\nonumber \\
& \sum_i ( -(1-x_i) M \kappa_i + (P_i - C_i) \pi_i + M z \lambda_i + B_i \tau_i + &\nonumber\\
& \gamma_i x_i Q ) + &\nonumber\\
& \nu (\sum_i x_i (P_i - C_i) + (1-z) M) & \label{eq:dualobjective1} \\
\mathrm{s.t. }\,\, & -\alpha_{ij}^p - \alpha_{ij}^c + \pi_j - \pi_i - \tau_i \geq 0  & (e_{ij})\label{cons:dualflowvar} \\
& \kappa_i + \gamma_i + \gamma_i + \theta_i + \lambda_i + \nu \leq 0 & (y_i) \\
& -\kappa_i + \pi_i - \theta_i + \tau_i \leq R_i & (r_i) \\
& -\pi_i \leq F_i & (f_i)\\
& \alpha^p, \alpha^c, \gamma, \theta, \nu, \lambda, \tau \leq 0, \kappa \geq 0.\label{cons:dualbounds}
\end{align}
\end{subequations}

\section{Multiloop compact model} \label{app:multiloop}

\begin{flushleft}
    
\adjustbox{width=1\linewidth}{%
\begin{minipage}{1.12\textwidth}
\input{models/multi-loop}
\end{minipage}
}
\end{flushleft}

The model presents symmetries similar to other clustering problems, any permutation of the loop assignments being equivalent. Constraint~\eqref{eq:symbreak} reduces the number of symmetric solutions by enforcing solutions with several actors decreasing with the loop index and Constraint~\eqref{eq:ml-oneloop} is added to ensure that each actor cannot belong to two different loops.
\newpage

\section{Dantzig-Wolfe decomposition of the multiloop compact model}\label{app:multiloop-danzig-wolfe}

We recall that $\Excessset_{st} \subseteq \Extendedloopset $ contains all loops $h$ that are net producers at scenario $s$ and time $t$. This fixed index set replaces the variable $z^{st}_l$ of the compact model. $D_{ij}$ is defined as $\min(P_i, C_j)$, resulting in Model~\eqref{model:MLExt}.

\input{models/dantzig-wolfe-multi-loop}

\newpage

\end{document}

%% file: config/packages.tex
\usepackage[english]{babel}

\usepackage[letterpaper,top=2cm,bottom=2cm,left=3cm,right=3cm,marginparwidth=1.75cm]{geometry}

\usepackage{ulem}
\usepackage{amsmath}
\usepackage{amsfonts}
\usepackage{graphicx}
\usepackage[colorlinks=true, allcolors=blue]{hyperref}

\usepackage{amsthm}
\usepackage[noabbrev,capitalize,nameinlink]{cleveref}[0.21]
\usepackage{bbm}

\usepackage{xspace}

\usepackage{multirow}
\usepackage{bigstrut}
\usepackage{booktabs}
\usepackage{colortbl}
\usepackage{tabularx}
\usepackage{etoolbox}
\usepackage{xcolor}
\usepackage{multicol}
\usepackage{adjustbox}
\usepackage{makecell}
\usepackage{soul}
\usepackage[nomessages]{fp}
\usepackage{hhline}
\usepackage{subcaption} 
\usepackage{setspace}
\usepackage{orcidlink}

%% file: config/commands.tex
\newtheorem{assumption}{Assumption}[section]
\newtheorem{proposition}{Proposition}[section]
\crefname{assumption}{Assumption}{Assumptions}

\newtheorem{definition}{Definition}[section]

\newcommand{\Actorsset}{\ensuremath{\mathcal{A}}\xspace}

\newcommand{\Edgeset}{\ensuremath{\mathcal{E}}\xspace}
\newcommand{\Scenarioset}{\ensuremath{\mathcal{S}}\xspace}
\newcommand{\Timeset}{\ensuremath{\mathcal{T}}\xspace}
\newcommand{\Compactloopset}{\ensuremath{\mathcal{L}}\xspace}
\newcommand{\Extendedloopset}{\ensuremath{\mathcal{H}}\xspace}
\newcommand{\Excessset}{\ensuremath{\mathcal{G}}\xspace}
\newcommand{\Vertexset}{\ensuremath{\mathcal{V}}\xspace}
\newcommand{\Neighbourset}{\ensuremath{\mathcal{N}}\xspace}
\newcommand{\Directset}{\ensuremath{\mathcal{D}}\xspace}

\newcommand{\rev}[1]  {{#1}}

%% file: tables/slr-criteria.tex
\begin{tabular}{|p{9em}p{39em}|}
\hline
\textbf{Criterion}     & \textbf{Description} \\
\hhline{==}
Problem size           & How many actors are considered? How are they dispatched on the territory? \bigstrut\\
\hline
Time range             & What are the optimisation time periods and horizon? \bigstrut\\
\hline
Method                 & Which method, algorithm, or model are used to solve the optimisation/simulation problem? \bigstrut\\
\hline
Self-consumption       & Is the collective aspect of self-consumption considered in the model? \bigstrut\\
\hline
Actors selection       & Does the model include the selection of actors and installations? \bigstrut\\
\hline
Leg. constraints       & Are some legal constraints integrated into the model? \bigstrut\\

\hline
\end{tabular}%

%% file: tables/slr.tex
\begin{tabular}{|l|lp{5.5em}lp{8.3em}p{6.665em}|}
\hline
\multirow{2}[4]{*}{Reference} & \multicolumn{3}{l|}{Size of the problem}                                 & \multicolumn{1}{l|}{\multirow{2}[4]{*}{Method}} & \multicolumn{1}{l|}{\multirow{2}[4]{*}{Actors selection}} \bigstrut\\
\cline{2-4}                       & \multicolumn{1}{p{4.165em}|}{nb actors} & \multicolumn{1}{p{5.5em}|}{time horizon} & \multicolumn{1}{l|}{time range} & \multicolumn{1}{l|}{}  & \multicolumn{1}{l|}{} \bigstrut\\
\hline
\cite{al-sorour_enhancing_2022}       & 6                      & 4 months               & 10 min                 & MILP                   & Yes \bigstrut\\
\hline
\cite{bahret_costoptimized_2021}       & 20                     & 1 year                 & NA                     & MILP                   & No \bigstrut\\
\hline
\cite{brusco_renewable_2023}          & 41                     & NA                     & hour                   & MILP                   & No \bigstrut\\
\hline
\cite{espadinha_assessing_2023}       & 4                      & 1 year                 & hour                   & Func minimisation  & No \bigstrut\\
\hline
\cite{gil_mena_analysis_2023}        & 12                     & 1 year                 & hour                   & Heuristic              & No \bigstrut\\
\hline
\cite{goitia-zabaleta_two-stage_2023} & 15                     & 1 year                 & NA                     & MILP                   & No \bigstrut\\
\hline
\cite{gulli_recoupled_2022}           & 4                      & 1 year                 & hour                   & MILP                   & No \bigstrut\\
\hline
\cite{luz_modeling_2021}         & 2                      & NR                     & 15 min                 & NA                     & No \bigstrut\\
\hline
\cite{mustika_new_2022}         & 7                      & 1 month                & 30 min                 & MINLP                  & No \bigstrut\\
\hline
\cite{perger_pv_2021}          & 15                     & 1 year                 & hour                   & MILP                   & No \bigstrut\\
\hline
\cite{pinto_optimization_2020}          & 50                     & NA                     & NA                     & MILP                   & No \bigstrut\\
\hline
\cite{reis_collective_2022}            & 50                     & NA                     & min                    & MAS - Heurisitic       & No \bigstrut\\
\hline
\cite{sangare_loads_2023}        & 7                      & 1 day                  & 30 min                 & MILP                   & No \bigstrut\\
\hline
\cite{simoiu_sizing_2021}          & 10                     & NA                     & hour                   & MILP                   & No \bigstrut\\
\hline
\cite{stephant_distributed_2021}        & 8                      & 1 day                  & NA                     & Game Theory            & No \bigstrut\\
\hline
Proposed model         & 100                    & 6 months               & hour                   & MILP - Benders Dantzig-Wolfe & Yes \bigstrut\\
\hline
\end{tabular}%

%% file: models/single-loop.tex
\begin{subequations}\label{model:compact}
\begin{align}
\min_{x, e, r, f, z, y} & \sum_{s\in\Scenarioset} \sum_{t\in\Timeset} \sum_{i\in [n]} p_s\, (F_i^{st} f_i^{st} - R_i^{st} r_i^{st}) & \\
\text{s.t. } & e_{ij}^{st} \leq x_i P_i^{st} & \forall (i,j) \in \Edgeset, s\in\Scenarioset,t\in\Timeset \label{eq:sl-actexcl1}\\
             & e_{ij}^{st} \leq x_j C_j^{st} & \forall (i,j) \in \Edgeset, s\in\Scenarioset,t\in\Timeset \label{eq:sl-actexcl2}\\
& \sum_{j\in \Neighbourset_i} e_{ij}^{st} + r_i^{st} + C_i^{st} = \sum_{j\in\Neighbourset_i} e_{ji}^{st} + f_i^{st} + P_i^{st} & \forall i \in [n], s\in\Scenarioset,t\in\Timeset \label{eq:sl-kir} \\
& y_i^{st} \leq x_i Q^{st} & \forall i \in [n], s\in\Scenarioset, t\in\Timeset \label{eq:sl-csc1}\\
& y_i^{st} \leq r_i^{st} & \forall i \in [n], s\in\Scenarioset, t\in\Timeset \label{eq:sl-csc2}\\
& y_i^{st} \geq r_i^{st} - (1-x_i) Q^{st}  & \forall i \in [n], s\in\Scenarioset, t\in\Timeset \label{eq:sl-csc3}\\
& \sum_{i\in [n]} y_i^{st} \leq \sum_{i\in [n]} x_i (P_i^{st} - C_i^{st}) + M^{st} (1-z^{st}) & \forall s\in\Scenarioset, t\in\Timeset \label{eq:sl-csc4}\\ 
& y_i^{st} \leq M^{st} z^{st} & \forall i \in [n], \forall s\in\Scenarioset, t\in\Timeset \label{eq:sl-csc5}\\
& \sum_{j \in \Neighbourset_i} e_{ij}^{st} + r_i^{st} \leq B_i^{st} & \forall i \in [n], s\in\Scenarioset,t\in\Timeset \label{eq:sl-noncir}\\
& x_i + x_j \leq 1 & \forall (i,j) \notin \Edgeset &\label{eq:sl-conflictgraphcons} \\
& \sum_{i\in [n]} x_{i} P^{\text{inst}}_i \leq P^{\text{inst}}_{\text{leg}} \label{eq:sl-power}\\
& z^{st} = 1 \Leftrightarrow  \sum_{i\in [n]} x_i (P_i^{st} - C_i^{st}) \geq 0 & \forall s\in\Scenarioset,t\in\Timeset \label{eq:sl-csc6}\\
& x \in \left\{0,1\right\}^{n}, z\in \left\{0,1\right\}^{|\Scenarioset|\, |\Timeset|} \label{eq:sl-compactbinvars} & \\
& e, r, f, y \geq 0.&
\end{align}
\end{subequations}

%% file: tables/selling_prices.tex
\begin{tabular}{|llllll|}
\hline
Installed power (kWc)  & $\le$ 3kWc & $\le$ 3-9kWc & 9-36kWc                & $\ge$ 36kWc & $\ge$ 100kWc \bigstrut\\
\hline
Selling price (c€/kWh)       & 13.39                  & 13.39                  & 14.58                  & 12.68                  & 13.12 \bigstrut\\
\hline
\end{tabular}%

%% file: tables/buying_prices.tex
\begin{tabular}{|lllll|}
\hline
\multicolumn{1}{|c}{\multirow{2}[4]{*}{Category}} & \multicolumn{2}{c}{$\le$ 36 kVA}   & \multicolumn{2}{c|}{$\ge$ 36kVA} \bigstrut\\
\cline{2-5}                       & Household              & Pro                    & \multicolumn{2}{c|}{Pro} \bigstrut\\
\hline
Period                 & \multicolumn{2}{c}{All year}                    & Winter                 & \multicolumn{1}{c|}{Summer} \bigstrut\\
\hline
Peak hours cost (c€/kWh)             & 20.4                   & 19.84                  & 27.26                  & 13.63 \bigstrut\\
\hline
Off-peak hours cost (c€/kWh)         & 15.13                  & 16.07                  & 15.16                  & 7.58 \bigstrut\\
\hline
\end{tabular}%

%% file: tables/plan.tex
\begin{tabular}{|llllll|}
\hline
Scenario               & Orientation            & Tilt                   & Density                & Installed power limit & Period \bigstrut\\
\hline
reference              & 180 (South)            & 30$^\circ$                     & 0.5 pros/km2           & 3 MWc & Winter \bigstrut\\
\hline
dens\_0\_1             & 180 (South)            & 30$^\circ$                     & 0.1 pros/km2           & 3 MWc & Winter \bigstrut\\
\hline
dens\_2                & 180 (South)            & 30$^\circ$                     & 2 pros/km2             & 3 MWc & Winter \bigstrut\\
\hline
period\_sum            & 180 (South)            & 30$^\circ$                     & 0.5 pros/km2           & 3 MWc & Summer \bigstrut\\
\hline
power\_1000            & 180 (South)            & 30$^\circ$                     & 0.5 pros/km2           & 1 MWc & Winter \bigstrut\\
\hline
power\_2000            & 180 (South)            & 30$^\circ$                     & 0.5 pros/km2           & 2 MWc & Winter \bigstrut\\
\hline
power\_5000            & 180 (South)            & 30$^\circ$                     & 0.5 pros/km2           & 5 MWc & Winter \bigstrut\\
\hline
config\_wc             & 0 (North)              & 60$^\circ$                     & 0.5 pros/km2           & 3 MWc & Winter \bigstrut\\
\hline
\end{tabular}%

%% file: tables/profiles_test.tex
\begin{tabular}{|l|l|c|c|}
\hline
Profile                & Reference consumption profile & Installed power range & Rate \bigstrut\\
\hline
Household              & \cellcolor[rgb]{ .867,  .922,  .969}Household (P $\le$ 36kVA) & \cellcolor[rgb]{ .867,  .922,  .969}0kW $\le$ P $\le$ 6kW & \cellcolor[rgb]{ .388,  .745,  .482}40\% \bigstrut\\
\hline
Pro1                   & \cellcolor[rgb]{ .867,  .922,  .969}Pro (P $\le$ 36kVA)   & \cellcolor[rgb]{ .867,  .922,  .969}6kW $\le$ P $\le$ 12kW & \cellcolor[rgb]{ .388,  .745,  .482}40\% \bigstrut\\
\hline
Pro2                   & \cellcolor[rgb]{ .867,  .922,  .969}Pro (P $\le$ 36kVA)   & \cellcolor[rgb]{ .608,  .761,  .902}1MW $\le$ P $\le$ 3MW & \cellcolor[rgb]{ .89,  .906,  .588}20\% \bigstrut\\
\hline
\end{tabular}%

%% file: tables/instances_size.tex
\begin{tabular}{l|lll|lll|lll}
\toprule
 & \multicolumn{3}{c|}{Constraints ($\times10^3$)} & \multicolumn{3}{c|}{Variables ($\times10^3$)} & \multicolumn{3}{c}{Integer variables} \\
Configuration & MLCol & MLCpct & SLCpct & MLCol & MLCpct & SLCpct & MLCol & MLCpct & SLCpct \\
\midrule
reference & 332 & 470 & 243 & 111 & 201 & 121 & 1198 & 6714 & 735 \\
dens\_0\_1 & 154 & 445 & 184 & 80 & 173 & 92 & 375 & 6529 & 735 \\
dens\_2 & 574 & 521 & 335 & 156 & 249 & 167 & 2168 & 7367 & 735 \\
period\_sum & 441 & 475 & 243 & 111 & 203 & 121 & 1198 & 6856 & 735 \\
power\_1000 & 127 & 475 & 243 & 110 & 203 & 121 & 831 & 6856 & 735 \\
power\_5000 & 431 & 475 & 243 & 111 & 203 & 121 & 1541 & 6856 & 735 \\
config\_wc & 226 & 475 & 243 & 110 & 203 & 121 & 1962 & 6856 & 735 \\
10\_act\_1\_d & 5 & 8 & 5 & 2 & 4 & 2 & 386 & 500 & 34 \\
10\_act\_10\_d & 45 & 73 & 47 & 21 & 34 & 23 & 386 & 1580 & 250 \\
10\_act\_14\_d & 63 & 102 & 66 & 29 & 48 & 33 & 386 & 2060 & 346 \\
10\_act\_30\_d & 138 & 218 & 140 & 62 & 102 & 70 & 386 & 3980 & 730 \\
10\_act\_60\_d & 284 & 434 & 281 & 124 & 204 & 140 & 386 & 7580 & 1450 \\
10\_act\_90\_d & 442 & 651 & 421 & 186 & 305 & 210 & 386 & 11180 & 2170 \\
10\_act\_180\_d & - & 1301 & 843 & - & 610 & 419 & - & 21980 & 4330 \\
10\_act\_270\_d & - & 1952 & 1264 & - & 914 & 629 & - & 32780 & 6490 \\
10\_act\_360\_d & - & 2602 & 1686 & - & 1219 & 839 & - & 43580 & 8650 \\
10\_act\_7\_d & 31 & 51 & 33 & 15 & 24 & 16 & 386 & 1220 & 178 \\
15\_act\_7\_d & 38 & 106 & 44 & 20 & 41 & 22 & 421 & 2137 & 183 \\
20\_act\_7\_d & 39 & 166 & 54 & 24 & 60 & 27 & 339 & 2860 & 188 \\
30\_act\_7\_d & 56 & 354 & 80 & 35 & 115 & 40 & 462 & 5165 & 198 \\
\bottomrule
\end{tabular}

%% file: tables/notations.tex
\begin{tabular}{|p{\linewidth}|}

\hline
\textbf{Sets} \\
\hline
$\Actorsset = \{A_{1}, \ldots, A_{n}\}$ all available actors that are potential members of a self-consumption loop\\
$\Scenarioset = \{s_{1}, \ldots, s_{n}\}$ the set of production and consumption scenarios\\
$\Timeset = \{t_{1}, \ldots, t_{n}\}$ the set time steps \\
$\Compactloopset = {l_{1}, \ldots, l_{n/2}}$ set of possible self-consumption loops having at least 2 actors each\\
$\Edgeset$ set of edges in the graph \\
$\Neighbourset_i$ set containing the neighbours of node $i$ \\
$\Extendedloopset$ set of maximal loops \\
$\Excessset_{st}$ set of loops having a positive net production for a given $(s,t)$ tuple \\
$\Vertexset$ set of vertices of the dual feasible set at scenario $s$ and time $t$ \\
$\Directset$ set of coupled actors\\ 

\hline
\textbf{Model Parameters} \\
\hline
$C_i$ net consumption of $A_i$ \\
$P_i$ net quantity of energy produced by $A_i$\\
$D_{i,j} = \min(P_i, C_j)$ \\
$F_i$ grid-sourced electricity cost for $A_i$\\
$R_i$ grid electricity selling price for $A_i$\\
$d_{ij}$ geographical distance between $A_i$ and $A_j$\\
$d_{\text{leg}}$ same loop actors maximal legal geographical distance\\
$P^{\text{inst}}_{i}$ installed power of $A_i$ \\
$P^{\text{inst}}_{\text{leg}}$ loop installed power legal limit\\
$p_{s}$ probability of occurrence of scenario $s$ \\
$Q$, $M$, $B_i$ Big M constants\\ 

\hline
\textbf{Models decision variables}  \\
\hline
$x_i$ binary variables indicating $A_i$ is a member of a self-consumption loop. In multiloop models, those variables are indexed by the loop: $x_i^l$\\
$w_{ij}^{l}$ binary variables linearising $x_{i}^{l} x_{j}^{l}$, for the membership of both actors $i$ and $j$ to loop $l$\\
$e_{ij}$ amount of energy exchanged from $A_i$ to actor $A_j$\\
$f_i$ amount of energy supplied by the grid to $A_i$\\
$r_i$ amount of energy sold to the grid by $A_i$\\
$z$ indicator variable ensuring loop over grid priority \\
$y_i$ auxiliary binary variable ensuring loop over grid priority \\
$v_h$ binary variable capturing if a loop is chosen\\
$a_i^h$ binary variable indicating whether loop $h$ contains actor $i$\\
$\alpha_{ij}^p$, $\alpha_{ij}^c$, $\pi_i$, $\gamma_i$, $\theta_i$, $\kappa_i$, $\nu$, $\lambda_i$, $\tau_i$ dual variables associated to primal constraints\\
$\eta$ objective of the main problem in Benders decompositions\\

\hline
\textbf{Other notations}  \\
\hline
$N_l$ number of proposed loops within a solution\\
$N_{\text{mean}}$ average number of loops over several instances\\
$P_{\text{mean}}$ average installed power of loops\\
$\overline{b}$ highest benefit realised among all actors through the loops\\
$\underline{b}$ lowest benefit realised among all actors through the loops\\
$\phi(x,z)$ solution to the binary subproblem in Benders decompositions \\
\hline
\end{tabular}

%% file: models/multi-loop.tex
\begin{subequations}\label{model:multiloop-cpct}
\begin{align}
\min_{x,z,w,e,r,f,z,y} & \sum_{s\in\Scenarioset} \sum_{t\in\Timeset} \sum_{i\in [n]} p_s\, (F_i^{st} f_i^{st} - R_i^{st} r_i^{st}) & \\
\text{s.t. }
& \sum_{j\in\Neighbourset_i} e_{ij}^{st} + r_i^{st} + C_i^{st} = \sum_{j\in\Neighbourset_i} e_{ji}^{st} + f_i^{st} + P_i^{st} & \forall i \in [n], s\in\Scenarioset,t\in\Timeset \\
& y_{li}^{st} \leq x_{i}^{l} Q^{st} & \forall i \in [n], s\in\Scenarioset, t\in\Timeset, l\in\Compactloopset \label{eq:scl-1}\\
& y_{li}^{st} \leq r_i^{st} & \forall i \in [n], s\in\Scenarioset, t\in\Timeset, l\in\Compactloopset \\
& y_{li}^{st} \geq r_{i}^{st} - (1-x_{i}^{l}) Q^{st} & \forall i \in [n], s\in\Scenarioset, t\in\Timeset, l\in\Compactloopset \\
& \sum_{i\in [n]} y_{li}^{st} \leq \sum_{i\in [n]} x_{i}^{l} (P_i^{st} - C_i^{st}) + M^{st} (1 - z_l^{st}) & \forall s\in\Scenarioset, t\in\Timeset, l\in\Compactloopset \\
& y_{li}^{st} \leq M^{st} z_l^{st} & \forall i \in [n], \forall s\in\Scenarioset, t\in\Timeset, l\in\Compactloopset \label{eq:scl-5}\\
& \sum_{i \in [n]} x_{i}^{l} P^{\text{inst}}_i \leq P_{\text{leg}}^{\text{inst}} & \forall l \in \Compactloopset\label{eq:knapsackmultiloop} \\
& x_{i}^{l} + x_{j}^{l} \leq 1 & \forall (i,j) \notin \Edgeset, l\in\Compactloopset \label{eq:ml-distance-constraint}\\
& \sum_{j\in\Neighbourset_i} e_{ij}^{st} + r_i^{st} \leq B_i^{st} & \forall i \in [n], s\in\Scenarioset,t\in\Timeset \label{eq:ml-circulation}\\
& \sum_{l\in\Compactloopset} x_{i}^{l} \leq 1 & \forall i \in [n] \label{eq:ml-oneloop}\\
& e_{ij}^{st} \le \sum_{l\in\Compactloopset} w_{ij}^{l} D_{ij}^{st} & \forall (i,j) \in \Edgeset, s\in\Scenarioset,t\in\Timeset \label{eq:consflowmultiloop} \\
& w_{ij}^{l} \le x_{i}^{l} & \forall (i,j) \in \Edgeset, \forall l \in \Compactloopset \label{eq:wij1}\\
& w_{ij}^{l} \le x_{j}^{l} & \forall (i,j) \in \Edgeset, \forall l \in \Compactloopset \label{eq:wij2}\\
& \sum_{i\in[n]} x_{i}^{l} \geq \sum_{i \in [n]} x_{i}^{l+1} & \forall l \in 1\dots |\Compactloopset|-1\label{eq:symbreak} \\
& z^{st}_l = 1 \Leftrightarrow  \sum_{i\in [n]} x_i^l (P_i^{st} - C_i^{st}) \geq 0 & \forall l \in \Compactloopset, s\in\Scenarioset,t\in\Timeset \label{eq:scl-6}\\
& x \in \left\{0,1\right\}^{n |\Compactloopset|}, z \in \left\{0,1\right\}^{|\Compactloopset \times \Scenarioset \times \Timeset|}, w\in \left\{0,1\right\}^{ |\Edgeset \times \Compactloopset|} & \\
& e, r, f, y \geq 0. &
\end{align}
\end{subequations}

%% file: models/dantzig-wolfe-multi-loop.tex
\begin{subequations}\label{model:MLExt}
\begin{align}
\min_{e,r,f,v} & \sum_{s\in\Scenarioset} \sum_{t\in\Timeset} \sum_{i\in [n]} p_s\, (F_i^{st} f_i^{st} - R_i^{st} r_i^{st}) & \\
\text{s.t. } & e_{ij}^{st} \leq D_{ij}^{st} \sum_{h\in\Extendedloopset} v_h a_i^h a_j^h & \forall (i,j) \in \Edgeset, s\in\Scenarioset,t\in\Timeset\\
& \sum_{j\in\Neighbourset_i} e_{ij}^{st} + r_i^{st} + C_i^{st} = \sum_{j\in\Neighbourset_i} e_{ji}^{st} + f_i^{st} + P_i^{st} &  \forall i \in [n], s\in\Scenarioset,t\in\Timeset \\
& r_i^{st} \leq (1-\sum_{h\notin \Excessset_{st} } v_h a_i^h ) Q^{st} & \forall i\in [n], s\in\Scenarioset, t\in\Timeset \\
& \sum_{j\in \Neighbourset_i} e_{ij}^{st} + r_i^{st} \leq P_{i}^{st} & \forall i\in [n], s\in\Scenarioset, t\in\Timeset \\
& \sum_{i\in[n]} r_i^{st} a_i^h \leq \sum_{i\in [n]} a_i^h v_h (P_i^{st} - C_i^{st}) + Q^{st} (1-v_h) & \forall s\in\Scenarioset, t\in\Timeset, h \in \Excessset_{st} \\
& \sum_{h\in\Extendedloopset} a_i^h v_h \leq 1 & \forall i \in [n]\\
& v \in \left\{0,1\right\}^{|\Extendedloopset|} \\
& e, r, f, y \geq 0.
\end{align}
\end{subequations}